\documentclass[11pt,a4paper]{amsart}
\usepackage[top=3.5cm,bottom=3cm,left=2.5cm,right=2.5cm]{geometry}
\usepackage[T1]{fontenc}
\usepackage[utf8]{inputenc}
\usepackage[english]{babel}
\usepackage{float}
\usepackage{xspace}
\usepackage{lmodern}
\DeclareFontFamily{U}{wncyr}{}
\DeclareFontShape{U}{wncyr}{m}{n}{<->wncyr10}{}
\DeclareFontShape{U}{wncyr}{m}{it}{<->wncyi10}{}
\DeclareFontShape{U}{wncyr}{m}{sc}{<->wncysc10}{}
\DeclareFontShape{U}{wncyr}{b}{n}{<->wncyb10}{}
\DeclareTextCommand{\guillemotleft}{T1}{%
  {\fontencoding{U}\fontfamily{wncyr}\selectfont\symbol{"3C}}%
}
\DeclareTextCommand{\guillemotright}{T1}{%
  {\fontencoding{U}\fontfamily{wncyr}\selectfont\symbol{"3E}}%
}
\usepackage{amsmath,amsfonts,amsthm,amssymb}
\usepackage{tikz-cd}
\usepackage{graphicx}
\usepackage{tikz}
\usepackage{ifthen}
\usetikzlibrary{arrows,calc,intersections,patterns}
\usepackage{rotating}
\usepackage{caption}
\usepackage{subcaption}
\usepackage{enumitem}

\allowdisplaybreaks[3] 
\usepackage{tikz-cd}
\usepackage{mathtools}
\usepackage{xcolor}

\newtheorem{theorem}{Theorem}[section]
\newtheorem{proposition}[theorem]{Proposition}
\newtheorem{corollary}[theorem]{Corollary}
\newtheorem{lemma}[theorem]{Lemma}

\theoremstyle{definition}

\theoremstyle{remark}
\newtheorem{remark}[theorem]{Remark}

\numberwithin{equation}{section}
\DeclareRobustCommand{\SkipTocEntry}[5]{}
\usepackage{siunitx}

\newcommand{\les}{\lesssim}

\usepackage[]{hyperref}
\hypersetup{
    colorlinks=true,       
    linkcolor=blue,          
    citecolor=magenta,        
    filecolor=magenta,      
    urlcolor=cyan           
}
\begin{document}

\title[Scattering for NLS]{Almost sure Scattering at Mass regularity for radial Schrödinger Equations}
\author{Mickaël Latocca}
\address{Département de Mathématiques et Applications, Ecole Normale Supérieure, 45 rue d'Ulm 75005 Paris, France}
\email{mickael.latocca@ens.fr}

\date{\today}

\maketitle

\begin{abstract}
\begin{sloppypar}
We consider the radial nonlinear Schrödinger equation $i\partial_tu +\Delta u = |u|^{p-1}u$ in dimension $d\geqslant 2$ for $p\in \left(1,1+\frac{4}{d}\right]$ and construct a natural Gaussian measure $\mu_0$ which support is almost $L^2_{\text{rad}}$ and such that $\mu_0$ - almost every initial data gives rise to a unique global solution. Furthermore, for $p>1+\frac{2}{d}$ and $d\leqslant 10$, the solutions constructed scatter in a space which is almost $L^2$. This paper can be viewed as a higher dimensional counterpart of the work of Burq and Thomann~\cite{BT19}, in the radial case.
\end{sloppypar}
\end{abstract}

\tableofcontents

\section{Introduction}

We consider the Cauchy problem and the long time dynamics of the nonlinear Schrödinger equation: 
\begin{equation}
    \label{3NLS}
    \tag{NLS}
    \left\{
    \begin{array}{c}
         i\partial _t u + \Delta u = |u|^{p-1}u \\
         u(0)=u_0 \in H^s(\mathbb{R}^d)\,, 
    \end{array}
    \right.
\end{equation}
where $p\geqslant 1$ need not be an integer, $s \in \mathbb{R}$ and $d \geqslant 2$. \eqref{3NLS} is known to be invariant under the scaling symmetry: \[u(t,x)\mapsto u_{\lambda}(s,y):=\lambda^{\frac{2}{p-1}}u(\lambda^2 s,\lambda y),\,\] which is such that $\|u_{\lambda}\|_{\dot{H}^s}=\lambda ^{s(d,p)}\|u\|_{\dot H^s}$ with $s(d,p):=\frac{d}{2}-\frac{2}{p-1}$, the \textit{critical} regularity threshold and where $\dot{H}^s$ refers to the homogeneous Sobolev space. \eqref{3NLS} also enjoys several formal conservation laws, the most important being the conservation of the \textit{mass} and the \textit{energy}, that is the quantities \[M(t):=\|u(t)\|_{L^2}^2 \text{ and } E(t):= \frac{1}{2}\|\nabla u(t)\|^2_{L^2}+\frac{1}{p+1}\|u(t)\|^{p+1}_{L^{p+1}}\]
are conserved under the flow of~\eqref{3NLS}. 

When $s>s(d,p)$, local well-posedness and even global well-posedness are expected and when~$s<s(d,p)$ an ill-posedness behaviour is to be expected. This article deals with the exponent range~$p \leqslant 1+ \frac{4}{d}$ and regularity $s=0$, \textit{i.e.}, the mass regularity. Remark that $s(d,1+\frac{4}{d})=0$. We gather below some known results concerning this problem. We recall that a solution $u$ to~\eqref{3NLS} is said to scatter in $H^s$ (which stands for the non-homogeneous Sobolev spaces) forward in time (resp. backward in time) if there exists $u_+ \in H^s$ (resp. $u_-$) such that:
\[\|u(t)-e^{it\Delta}u_{\pm}\|_{H^s} \underset{t \to \pm\infty}{\longrightarrow} 0\,.\]  
This expresses that even if $u$ is the solution of a nonlinear equation, its long time behaviour will eventually be close the linear evolution of the state $u_{\pm}$, which may not always be the initial data $u_0$. 

\subsection{Scattering for the nonlinear Schrödinger equations}

The literature pertaining to the long time behaviour of~\eqref{3NLS} is broad and we do not claim to be exhaustive. For an introduction to the deterministic theory of the mass sub-critical and critical nonlinear Schrödinger equations we refer to~\cite{taoDispersive,cazenave}. The following theorem gathers some of the most important results known in the scattering theory of~\eqref{3NLS}. We also refer to \cite{nakanishi, dodsonD1} for more scattering results.  

\begin{theorem}[Deterministic theory]\label{3theoDeterministic} Let $p \geqslant 1$ and $d\geqslant 2$. Then: 
\begin{enumerate}[label=(\textit{\roman*})]
    \item For $p \in \left[1,1+\frac{4}{d}\right]$ the Cauchy problem for~\eqref{3NLS} is globally well-posed in $L^2(\mathbb{R}^d)$ and if $p> 1+ \frac{4}{d}$ the Cauchy problem is ill-posed in $L^2$. 
    \item For $p \leqslant 1+\frac{2}{d}$ and for every $u_0 \in L^2$, the solutions do not scatter in $L^2$, neither forward nor backward in time. 
    \item For $d \geqslant 2$, $p \in \left(1+\frac{2}{d},1+\frac{4}{d-2}\right)$ and initial data in $H^1$ scattering in $L^2$ holds. 
    \item For $d \geqslant 2$, $p=1+\frac{4}{d}$ and initial data in $L^2$, scattering holds in $L^2$.
\end{enumerate}
\end{theorem}

\begin{proof} (\textit{i}) is the standard local well-posedness result when $p<1+\frac{4}{d}$, see~\cite{taoDispersive}, Chapter~3 and globality comes from the conservation of mass. In the critical case the global theory is more involved, see~\cite{dodsonD1}. The ill-posedness part is proven in~\cite{christCollianderTao,alazardCarles}. (\textit{ii}) is the content of~\cite{barab}. For (\textit{iii}) see \cite{tsutsumiYajima} and for (\textit{iv}) see~\cite{taoScat}. For dimension $2$ see~\cite{dodsonD1,dodsonD2}. 
\end{proof}

The results of Theorem~\ref{3theoDeterministic} do not address scattering for $d\geqslant 1$ and $p \in (1+2/d, 1+4/d]$, whereas the scaling heuristics suggests it. To the author's knowledge, this has not been obtained by deterministic means. One way of attacking this problem is to study the Cauchy problem for~\eqref{3NLS} with random initial data. 
The theory of dispersive equations with random initial data can be tracked back at least to the pioneer work of \mbox{Bourgain}~\cite{bourgain,bourgain2d} and the use of formal invariant measures. See also~\cite{burqTzvetkov2} and~\cite{tzvetkov1, OhTadahiTzvetkov, ohTzvetkov1, ohSosoeTzvetkov} for an approach with quasi-invariant measure constructions. A key feature of the random data dispersive equation theory is that it constructs solutions at low regularity. 

Scattering at mass regularity for $p\in (3,5]$ in dimension $1$ with random initial data has been partially addressed in~\cite{burqThomannTzvetkov}, in which the authors prove almost sure global existence and scattering at almost $L^2$ regularity with respect to a measure whose typical regularity is $L^2$, as long as $p \geqslant 5 = 1 + \frac{4}{d}$ and thus missed the range $(3,5)$. Their method is based on the use of the \textit{lens transform} (see Appendix~\ref{3appendixLens} for details) which transforms the scattering problem for~\eqref{3NLS} into a scattering problem for an harmonic oscillator version of the nonlinear Schrödinger equation, which turns out to be more amenable. The range $p\in(3,5)$ has been recently obtained in~\cite{BT19} by studying quasi-invariant measures in a quantified manner. In both cases such results are interesting in the sense that they give large data scattering, without assuming decay at infinity.
Let us mention that the construction of an invariant measure in the infinite volume setting of $\mathbb{R}^d$ was previously considered by Bourgain in~\cite{bourgain00}, in which invariant measures are constructed for~\eqref{3NLS} posed on $[-L,L]$ before taking the limit $L \to \infty$.

In dimension $d=2$, the counterpart of~\cite{burqThomannTzvetkov} in the radial case is established in~\cite{deng}. See also~\cite{poiretRobertThomann,thomann} where almost sure scattering in higher dimension was studied (although a smallness assumption is required). 

Note that in another setting, scattering for the nonlinear wave equations at energy regularity has been studied, and we refer to the works~\cite{dodsonLuhrmannMendelson} and also~\cite{bringmann1,bringmann2}.  

\subsection{Notation} We adopt widely used notations such as $\lfloor x \rfloor$ for the lower integer part of a real number $x$ and $\langle x \rangle := (1+|x|^2)^{\frac{1}{2}}$. We denote $\otimes$ the tensor product (of functions, spaces or measures) and we write $[A,B]=AB-BA$ the commutator of $A$ and $B$. The letters $\Omega$ and $\mathbb{P}$ will always denote a probability space and its associated probability measure with an expectation denoted by $\mathbb{E}$. 

For inequalities we often write $A \lesssim B$ when there exists a universal constant $C>0$ such that $A \leqslant CB$. In some cases we need to track explicitly the dependence of the constant $C$ upon other constants and we denote by $A \lesssim_a B$ to indicate that the implicit constant depends on $a$. In other cases we use the notation $C$ for a constant which can change from one line to another, and write $C(a,b)$ to explicitly recall the dependence of $C$ on other parameters. We write $A \sim B$ if $A \les B$ and $A \gtrsim B$. 

$D_t$ is defined as $-i\partial _t$. $\mathcal{S}$ and $\mathcal{S}'$ respectively stand for the space of Schwartz functions and its dual the space of tempered distributions. The index $\text{rad}$ means radial, so for example $\mathcal{S}_{\text{rad}}$ stands for the radially symmetric functions of $\mathcal{S}$. The Lebesgue spaces are denoted by $L^p$ and for $p \in [1,\infty]$ we set $p'$ such that $\frac{1}{p}+\frac{1}{p'}=1$. If $X$ is a Banach space then $L^p_TX$ serves as a shorthand for $L^p((0,T),X)$ and $L^p_w$ denotes the Lebesgue space with weight $w$. The usual Sobolev spaces are denoted by $H^s=W^{s,2}$ where $W^{s,p}=\{u \in \mathcal{S}', \; (\operatorname{id}-\Delta)^{\frac{s}{2}}\in L^p\}$. Let $H \coloneqq -\Delta + |x|^2$ be the Harmonic oscillator, we define 
\begin{equation*}
    \label{eq:defWsp}
    \mathcal{W}^{s,p}=\{u, H^{\frac{s}{2}}u\in L^p\}\,,
\end{equation*}
and
\begin{equation*}
    \label{eq:defHs}
    \mathcal{H}^s\coloneqq \mathcal{W}^{s,2}\,,
\end{equation*}
and these spaces are called the harmonic Sobolev spaces. The Hölder space $C^{0,\alpha}(I,X)$, for $\alpha \in(0,1)$, is defined as the set of continuous functions on $X$ that satisfy 
\[
    \|f\|_{C^{0,\alpha}}\coloneqq \|f\|_{L^{\infty}(I,X)}+\sup_{t\neq s \in I} \frac{\|f(t)-f(s)\|_X}{|t-s|^{\alpha}}<\infty\,.
\]

Smooth Littlewood-Paley spectral projectors for the Hermite operator at frequency $N=2^n$ are denoted by $\mathbf{P}_N$ and we set~$\mathbf{S}_N:=\sum_{m=0}^n \mathbf{P}_{2^m}$ and $\mathbf{P}_{>N}=\operatorname{id}-\mathbf{S}_N$. Truncation in frequency in $[-N,N]$ will be denoted by $\Pi_N$. We refer to Section~3.3.2 of~\cite{BT19} for precise definitions.  

We denote by $B_Z(\lambda)$ the closed ball of the space $Z$, centred at $0$ and of radius $\lambda$. 

If $a_n, b_n$ are real independent standard Gaussian random variables then $g_n:=a_n+ib_n$ is called a complex Gaussian random variable. 

\subsection{Main results}

This section presents the main results we shall prove. Precise definitions of the measure $\mu_0$ will be given in Remark~\ref{rem.defmu}. At this stage one can picture $\mu_0$ as a measure supported by radially symmetric functions which are almost $L^2$. 

Keeping in mind Theorem~\ref{3theoDeterministic}~(\textit{ii}), when $p \leqslant 1+\frac{2}{d}$ there is no scattering in $L^2$ and thus even in the probabilistic approach there is no possible scattering result. However we have global well-posedness. Let us introduce the parameter: 
\begin{equation}
    \label{eq:alphaDef}
    \alpha(p,d)\coloneqq 2-\frac{d}{2}(p-1) \geqslant 0\,.
\end{equation}

\begin{theorem}[Almost-sure global existence and weak scattering]\label{3side} Let $d \geqslant 2$ and $p \in \left(1,1+\frac{4}{d}\right]$. There exists $\sigma \in (0,\frac{1}{2})$ such that for $\mu_0$-almost every initial data $u_0$, there exists a unique global solution $u=e^{it\Delta}u_0 + w$ to~\eqref{3NLS} in the space
\[
    e^{it\Delta}u_0 + \mathcal{C}(\mathbb{R},H^{\sigma}(\mathbb{R}^d))\,.
\]

This solution satisfies
\[
    \|w(s)\|_{H^{\sigma}}\lesssim \langle s \rangle^{\frac{\alpha (p,d)}{2}} \log^{\frac{1}{2}} \langle s\rangle \,,
\]
for all $s \in \mathbb{R}$. 

Furthermore if $p \leqslant p_{\operatorname{max}}(d)$ (defined by~\eqref{eq:defPmax1} and \eqref{eq:defPmax2}) then 
\[
    \|u(s)\|_{L^{p+1}} \leqslant C(u_0) \log^{\frac{1}{2}}(1+|s|)\,,
\] 
for any $s\in \mathbb{R}$. 

Finally, in the above estimates, the constants $C(u_0)$ satisfy the following: there exist constants $C,c>0$ such that 
\[
    \mu_0 (u_0: C(u_0)> \lambda) \leqslant C e^{-c \lambda ^2}\,.
\]
\end{theorem}

For $p \in \left(1+\frac{2}{d},1+\frac{4}{d}\right]$ and $d \leqslant 10$ scattering holds in $L^2$ for almost every radial initial data at mass regularity. More precisely we define the following regularity parameter:
\begin{equation}
\label{3eqSigma}
\sigma(p,d)\coloneqq\left\{
\begin{array}{cl}
     \frac{1}{2} & \text{ if } p \leqslant 1+\frac{3}{d-2}  \\
     2-\frac{d-2}{2}(p-1)& \text{ if } p \geqslant \frac{3}{d-2}\,.
\end{array}
\right.
\end{equation}

\begin{theorem}[Almost sure scattering at mass regularity]\label{3main} Let $d \in \{2, \dots, 10\}$ and $p \in \left(1+\frac{2}{d}, 1+\frac{4}{d}\right]$. 
\begin{enumerate}[label=(\textit{\roman*})]
    \item For every $\sigma \in (0,\sigma(p,d))$ and $\mu_0$-almost every initial data there exists a unique global solution to~\eqref{3NLS} satisfying 
    \[
        u(s)-e^{is\Delta}u_0 \in \mathcal{C}(\mathbb{R}, H^{\sigma}(\mathbb{R}^d))\,.
    \]   
    \item The solutions constructed scatter at infinity, in $L^2$. More precisely, there exist $\sigma \in (0,\sigma (p,d))$ and $\kappa >0$ only depending on $p$ and $d$ such that for $\mu_0$-almost every $u_0$ there exist $u_{\pm} \in \mathcal{H}^{\sigma}$ (defined in~\eqref{eq:defHs}) such that the corresponding solution constructed in~(\textit{i}) satisfies
    \begin{equation}
        \label{3scat1}
        \|u(s)-e^{is\Delta}(u_0+u_{\pm})\|_{\mathcal{H}^{\sigma}} \lesssim C(u_0) \langle s\rangle^{-\kappa} \underset{s \to \pm \infty}{\longrightarrow} 0\,,
    \end{equation}
    and also 
    \begin{equation}
        \label{3scat2}
        \|e^{-is\Delta}u(s)-(u_0+u_{\pm})\|_{\mathcal{H}^{\sigma}} \lesssim C(u_0)\langle s\rangle^{-\kappa} \underset{s \to \pm \infty}{\longrightarrow} 0\,.
    \end{equation}
    In both cases there exist numerical constants $C,c>0$ such that
    \[
        \mu_0 (u_0:C(u_0)>\lambda) \leqslant Ce^{-c\lambda}\,.
    \]
\end{enumerate}
\end{theorem}

\begin{remark} Since the corresponding measure $\mu_0$ will be such supported by the radial functions of $\bigcap_{\sigma >0}H^{-\sigma}_{\text{rad}}$ we see that this result is radial in nature. We also emphasis that the convergence rate $\kappa$ can be made explicit by following the computations line by line.
\end{remark}

\begin{remark} First note that in dimension $d=2$, the case $p=3$ has been treated in~\cite{deng}. Similarly in dimension $d>2$ the endpoint case $p=1+\frac{4}{d}$ is an adaptation of the proof in~\cite{deng} and we do not treat this case in view of the deterministic result~\cite{taoScat}. Then we assume that $p \in \left(1+\frac{2}{d},1+\frac{4}{d}\right)$. 
\end{remark}

\begin{remark} The limitation $d\leqslant 10$ is due to the fact that as soon as $d \geqslant 3$, the obtained almost-sure local Cauchy theory combined with the Sobolev embedding are not powerful enough to control the $L^{p+1}$ norms when $p$ is close to $1 + \frac{4}{d}$. In order to control these norms one can use dispersive estimates instead of the Sobolev embedding. The dispersive estimate becomes however weaker when $d$ grows, leading to the limitation $d \leqslant 10$.
\end{remark}

\begin{remark} This result is not of small data type. The measure $\mu_0$ indeed satisfies that for $p>2$ sufficiently close to $2$ and every $R>0$, $\mu_0 (u\,, \|u\|_{L^p}>R)>0$, see Section~\ref{3secProba} for a proof. 
\end{remark}

\subsection{Strategy of the proof and organisation of the paper} 
The following paragraphs outline the proof of the main results. There are three main features that we need to address: assuming a global theory, how can one prove scattering? How can one construct a good local theory? How to extend this local theory into a powerful enough global theory? 

\subsubsection{Proving scattering} Assume that one has access to an almost-sure local existence theory for~\eqref{3NLS}, even a global theory in $L^2$. 

In order to deal with the long-time behaviour of the solutions $u$ we write, thanks to Duhamel's formula: 
\[
    u(s)=e^{is\Delta}u_0-i\int_0^{s}e^{i(s-s'))\Delta}\left(|u(s')|^{p-1}u(s')\right) \,\mathrm{d}s'\,,
\] 
so that in order to prove the existence of $u_+$ we only need to prove that the integral 
\[
    \int_0^{s}e^{i(s-s')\Delta}\left(|u(s')|^{p-1}u(s')\right) \,\mathrm{d}s'
\] 
converges, which is achieved by proving that this integral actually converges absolutely. From there we can see that \textit{a priori} bounds on $u$ in $L^r_x$ spaces may be needed to carry such a program. Such bounds can be obtained using the Sobolev embedding if $u$ is more regular. This heuristic suggests that we can try to construct a better local theory in a probabilistic setting, building on some stochastic smoothing. We then will need to extend the theory to a global one. 

\subsubsection{Almost-sure local well-posedness} The standard method to prove local well-posedness is to implement a fixed-point argument in some Banach spaces. In the context of dispersive equations, Strichartz estimates allow for achieving such a goal. Heuristically, Strichartz estimates tell us that if $u_0 \in H^s$ then the free evolution $u_L(s)=e^{is\Delta}u_0$ may not be smoother than $u_0$ but however exhibits some gain in space-time integrability, from which we access to a wider range of $(p,q)$ such that bounds of the form $\|u_L\|_{L^p_TL^q_x} \lesssim \|u_0\|_{L^2}$ hold. The advantage of working with random data is that improving the $L^p$ integrability of a random $L^2$ function (on the torus, for example) is essentially granted by a result which appeared in the work of Paley-Zygmund. See~\cite{burqTzvetkov2} Appendix~A for a proof. 

\begin{theorem}[Kolmogorov-Paley-Zygmund]\label{3theoremCentral} Let $(c_n)_{n \in \mathbb{Z}}$ an $\ell^2$ sequence. Then if $(g_n)_n$ is a sequence of identically distributed centred and normalised complex Gaussian variables we have 
\[
    \left\|\sum_{n \in \mathbb{Z}} c_ng_n\right\|_{L^p_{\Omega}} \lesssim \sqrt{p} \|(c_n)\|_{\ell^2}\,.
\]
\end{theorem}

Using this result we can prove a probabilistic Strichartz estimate which improves the classical one. Take $u_0$ a random initial data, then with the previous remarks we can easily control the space-time norms of $u_L(s)$, so we seek solutions of the form $u(s)=u_L(s)+w(s)$ where $w$ is deterministic and can be taken in a smoother space, for example $w \in H^s$ for some $s>0$. Formally $w$ is the solution of the fixed point problem $\Phi (w)=w$ where \[\Phi (w)=i\int_0^se^{i(s-s')\Delta} \left((u_L(s)+w(s))^p\right)\] and thanks to the gain in controlling the space time norms of $u_L$ we can expect to solve this problem in $H^s$. For an illustration of the method see~\cite{burqTzvetkov2}. For probabilistic well-posedness of~\eqref{3NLS} below the scaling regularity see~\cite{benyiOhPocovnicu}. 

In our context we will have to take care of the fact that since we will work with random initial data slightly below $L^2$ we will not really access to all the range of Strichartz estimates. These statements are made precise by Lemma~\ref{3lemmaGainSobolev}. Establishing such a good local theory is the content of Proposition~\ref{3propLocal}. 

\subsubsection{The globalisation argument.} Once one has a good local well-posedness theory, there exists a general globalisation argument which can be traced back at least to Bourgain~\cite{bourgain}. In order to clarify the arguments of Proposition~\ref{prop:globalWP} and Proposition~\ref{3lemGainLpp1} we explain the argument in a simpler setting, the invariant measure setting. One wants to achieve an almost sure global theory. To this end we first remark that thanks to the Borel-Cantelli lemma it is sufficient to prove that for every $\delta >0$ and every $T>0$ there exists a set $G_{\delta , T}$ such that $\mu (X\setminus G_{\delta, T}) \leqslant \delta$ and that for every initial data in $G_{\delta,T}$ we have existence on $[0,T]$. The requirements for the argument to work are roughly the following: 

\begin{enumerate}[label=(\textit{\roman*})]
    \item A local well-posedness theory of the following flavour: for initial data at time $t_0$ of size (in a space $X$) less than $\lambda$ there exists a solution on $[t_0,t_0+\tau]$ with $\tau \sim \lambda^{-\kappa}$, in a space~$X$.
    \item An invariant measure $\mu$, that is if $\phi_t$ denotes the flow of the equation, $\mu (A)=\mu(\phi_{-t}A)$ for all $t>0$ and every measurable set $A$, at least formally. 
\end{enumerate}

We further assume that $\mu(u_0, \|u_0\|_X>\lambda) \lesssim e^{-c\lambda^2}$, as this will always be the case in the following. 

Let $\lambda >0$ to be chosen later. Let $G_{\lambda}$ be the set of \textit{good} initial data $u_0$ which give rise to solutions defined on $[0,T]$ and such that $\|u(t)\|_X \leqslant 2 \lambda$ for every $t \leqslant T$. Let also $A_{\lambda}:=\{u_0, \; \|u_0\| \leqslant \lambda\}$. Choose $\tau$ small enough (which amounts to enlarging $\lambda$) so that the local theory constructed implies that a solution on $[n\tau, (n+1)\tau]$ does not grow more than doubling its size. Then we immediately see that \[\bigcap _{n=0}^{\lfloor \frac{T}{\tau}\rfloor} \phi_{n\tau}^{-1}(A_{\lambda}) \subset G_{\lambda}\,.\]
Taking the complementary we see that 
\[\{u_0 \notin G_{\lambda}\} \subset \bigcup _{n=0}^{\lfloor \frac{T}{\tau}\rfloor} \phi_{n\tau}^{-1}(X \setminus A_{\lambda})\,,\] and taking into account the expression of $\tau$ and the fact that $\mu$ is invariant under the flow this gives \[\mu (X\setminus G_{\lambda})\leqslant CT\lambda^{\kappa}e^{-c\lambda^2}\,,\] which is smaller than $\delta$ for $\lambda \sim \log^{\frac{1}{2}} \left( \frac{T}{\delta}\right)$. 

From there we obtain the global existence and a logarithmic bound for $\|u(t)\|_X$. However let us recall that there is no non-trivial invariant measure for the linear Schrödinger equation, or the nonlinear equation in presence of scattering in $L^2$.

\begin{proposition} Let $d \geqslant 2$. Then,
\begin{enumerate}[label=(\textit{\roman*})]
    \item Let $\sigma \in \mathbb{R}$. The only measure supported in $H^{\sigma}$ which is invariant by the flow of the linear Schrödinger equation $i\partial _s u + \Delta_y u =0$ with initial condition $u(0) = u_0 \in H^{\sigma}(\mathbb{R}^d)$ is $\delta_0$.
    \item Let $p$ such that scattering holds in $L^2$ for solutions of~\eqref{3NLS}. Then for any $\sigma \in \mathbb{R}$ the only measure supported in $L^2$ which is invariant by the flow of~\eqref{3NLS} is $\delta_0$.
\end{enumerate}
\end{proposition}

\begin{proof} We refer to \cite{BT19} where a proof is given in dimension $1$. The proof in dimension $d$ is a straightforward adaptation. 
\end{proof}

In order to overcome this difficulty, the authors in~\cite{burqThomannTzvetkov} use the \textit{lens transform} (see Appendix~\ref{3appendixLens}) to transform solutions of~\eqref{3NLS} with time interval $\mathbb{R}$ to solutions of
\begin{equation}
    \label{3NLSH}
    \tag{HNLS}
    \left\{
    \begin{array}{c}
         i\partial _t v - Hv = \cos(2t)^{-\alpha(p,d)}|v|^{p-1}v \\
         v(0)=u_0 \in \mathcal{H}^s \,, 
    \end{array}
    \right.
\end{equation}
with time interval $(-\frac{\pi}{4},\frac{\pi}{4})$ where we recall that $H=-\Delta + |x|^2$ is the harmonic oscillator, and $\mathcal{H}^s$ stands for the associate Sobolev spaces and $\alpha(p,d)$ has been defined in~\eqref{eq:alphaDef}. 

\eqref{3NLSH} admits quasi-invariant measures with quantified evolution bounds, see Proposition~\ref{3propMeasureEvol}, which makes it possible to use a similar globalisation argument to the one presented. For this reason most of this paper studies properties of~\eqref{3NLSH}, and we will eventually get back to~\eqref{3NLS} using the \textit{lens transform} in Section~\ref{3secEnd}. 

For the global theory, the main difference when compared to~\cite{BT19} lies in the proof of Proposition~\ref{3lemGainLpp1}. Let us recall that the proof in~\cite{BT19} uses: 
\begin{enumerate}[label=(\textit{\roman*})]
    \item A pointwise bound in $L^{p+1}$ with large probability for the solutions. 
    \item A Sobolev embedding to control the $L^{p+1}$ norm variation by the $\mathcal{H}^s$ norm variation, controlled by the local theory. 
\end{enumerate}
In dimension $d \geqslant 3$ such a program would not yield the full proof of Theorem~\ref{3main}, due to the use of the Sobolev embedding which would only yield a result for $p \leqslant 1 + \frac{2}{d-1}$. Instead we control the variation of the $L^{p+1}$ norm, writing: 
\begin{equation}
    v(t')=e^{-i(t'-t_n)H}v(t_n) - i \int_{t_n}^{t'} e^{-i(t'-s)H} \left(|v(s)|^{p-1}v(s)\right) \,\mathrm{d}s\,.
\end{equation}
for $t'$ in some interval $[t_n,t_{n+1}]$. The first term is controlled using linear methods. For the second term we need to control its $L^{\infty}L^{p+1}$ norm, which is not always Schrödinger admissible (terminology defined in Section~\ref{3secLocal}). Instead we use the dispersive estimate $L^{\frac{p+1}{p}} \to L^{p+1}$. A similar idea was previously used in~\cite{pocovnicuWang}. In the case $d\in\{9,10\}$ and values of $p$ close to $1+\frac{4}{d}$, we use a different method to obtain growth bounds on $\|v(t)\|_{L^{p+1}}$, see Proposition~\ref{3propMain}.

Once we have these $L^{p+1}$ bounds, scattering in $\mathcal{H}^{\varepsilon}$ for~\eqref{3NLSH} is obtained through interpolation between scattering in $\mathcal{H}^{-\sigma}$ (obtained by Sobolev embeddings and the $L^{p+1}$ bounds) and a bound for the solution in $\mathcal{H}^{\varepsilon}$. See Proposition~\ref{3propMain} for the details.  

\subsubsection{Organisation of the paper} The remaining of this paper is organised as follows: Section~\ref{3secProba} introduces the Gaussian measure $\mu$ and gathers their properties. Then the proof of Theorem~\ref{3side} and Theorem~\ref{3main} starts in Section~\ref{3secLocal} with a large probability local well-posedness theory. The evolution of the quasi-invariant measures are studied in Section~\ref{3secEvol} which allows to extend the local theory into an almost-sure global theory in Section~\ref{3sec:global} and prove Theorem~\ref{3side} and Theorem~\ref{3main}. Finally technical estimates are gathered in Appendix~\ref{3appTechnical} and Appendix~\ref{3appendixLens} for results concerning the \textit{lens transform}. 

\subsection*{Acknowledgements} The author thanks Nicolas Burq and Isabelle Gallagher for suggesting this problem and subsequent discussions. The author thanks Nicola Visciglia for raising the question of the extension of the results in \cite{burq} to higher dimensions and Paul Dario for useful comments on a draft of the manuscript. 

\section{The harmonic oscillator and quasi-invariant measures}\label{3secProba}

\subsection{The harmonic oscillator in the radial setting} 

We recall some facts concerning the radial harmonic oscillator, and refer to \cite{szego} for proofs. 

The radial harmonic oscillator is defined as $H =- \Delta + |x|^2$ acting on the space $\mathcal{S}_{\text{rad}}(\mathbb{R}^d)$ or radial Schwartz functions. It is a symmetric operator and admits a self-adjoint extension to $\mathcal{H}_{\text{rad}}^{1}(\mathbb{R}^d)$. 

It is known that the spectrum of $H$, acting on $\mathcal{H}^1_{\text{rad}}(\mathbb{R}^d)$ is made of eigenvalues 
\[
    \lambda^2_n=4n+d \text{ for } n\geqslant 0\,,
\]
and associated (simple) eigenfunctions are denoted by $e_n$. All that will be used in the rest of this article is that the sequence $(e_n)_{n\geqslant 0}$ satisfies the following $L^p$ estimates. We also refer to~\cite{imekrazRobertThomann} for details on the harmonic oscillator acting on radial functions.

\begin{lemma}[\cite{imekrazRobertThomann}, Proposition~2.4]\label{3lemmaGainLp} Let $d \geqslant 2$. Then
\begin{enumerate}[label=(\textit{\roman*})]
    \item $\|e_n\|_{L^p} \lesssim \lambda_n^{-d\left(\frac{1}{2}-\frac{1}{p}\right)}$ for $ p \in \left[2,\frac{2d}{d-1}\right)$;
    \item $\|e_n\|_{L^p} \lesssim \lambda_n^{- \frac{1}{2}} \log^{\frac{1}{p}} \lambda _n$ for $p=\frac{2d}{d-1}$;
    \item $\|e_n\|_{L^p} \lesssim \lambda_n^{d\left(\frac{1}{2}-\frac{1}{p}\right)-1}$ for $p \in \left(\frac{2d}{d-1},\infty\right]$. 
\end{enumerate}
In the above inequalities the implicit constant may depend on $d$ but not $p$ nor $n$. 
\end{lemma}

\begin{remark} From this lemma we see that in the scale of $L^p$ regularity, as long as $p<\frac{2d}{d-2}$, then eigenfunctions $e_n$ exhibit some decay, which is used to prove probabilistic smoothing estimates, see for example~\cite{burqThomannTzvetkov}, Appendix~A for such estimates.
\end{remark}

\subsection{Measures}\label{3sectionProba}

We fix a probability space $(\Omega, \mathcal{A}, \mathbb{P})$ and let $(g_n)_{n\geqslant 0}$ be an identically distributed sequence of centred, normalised complex Gaussian variables. Then we define the random variable $f:\Omega \to\mathcal{S}'(\mathbb{R}^d)$ by
\begin{equation}
    \label{3randSetup}
    f^{\omega}(x) \coloneqq \sum_{n \geqslant 0} \frac{g_n(\omega)}{\lambda _n} e_n(x)\,,
\end{equation}
where $\lambda _n, e_n$ are the eigenvalues and eigenfunctions of the radial harmonic oscillator previously defined. Moreover we consider the approximations \[
    f_N \coloneqq\sum _{n=0}^N \frac{g_n}{\lambda _n}e_n\,,
\]
which almost surely converge to $f$ in $\mathcal{S}'$. Moreover $(f_N)_{N \geqslant 1}$ is a Cauchy sequence in $L^2(\Omega,\mathcal{H}^{-\sigma})$ for every $\sigma >0$. For a proof of this fact see \cite{burqThomannTzvetkov}, Lemma~3.3. More precisely, if we denote by $\mu \coloneqq f_{*}\mathbb{P}$, the law of the random variable $f$, and set 
\[
    X\coloneqq\displaystyle \bigcap_{\sigma >0} \mathcal{H}^{-\sigma}_{\operatorname{rad}}\,.
\]

We have the following. 

\begin{lemma} The measure $\mu$ is supported by $X$ and moreover:
\begin{enumerate}[label=(\textit{\roman*})]
    \item For $\mu$ almost every $u$ one has~$u \notin L^2$. 
    \item For any $p > 2$ sufficiently close to $2$ and any $R>0$, $\mu (u \in X, \; \|u\|_{L^p}>R)>0$. 
\end{enumerate}
\end{lemma}

\begin{remark}\label{rem.defmu} If $\mu_0$ denotes measure given by $\mathcal{L}^{-1}_*\mu$, where $\mathcal{L}$ is the \textit{lens transform} then we infer that $\mu_0$ is supported by $\bigcap_{\sigma >0}H^{-\sigma}_{\text{rad}}$ and satisfies (\textit{ii}). 
\end{remark}

\begin{proof} (\textit{i}) The first part is a straightforward adaptation of Lemma~3.3 in~\cite{burqThomannTzvetkov}. For the proof of the fact that for $\mu$-almost every $u$, there holds $u \notin L^2$, see \cite{poiret}, Section~4, Lemma~53 and Proposition~54. 

(\textit{ii}) This result is claimed in~\cite{burqThomannTzvetkov}. We provide a proof. We construct large data in the following way. Let $E \subset \Omega$ be the set defined by:
\[
    E\coloneqq\left\{\omega \in \Omega, g_0^{\omega} >\frac{2\lambda_0 R}{\|e_0\|_{L^{p}}} \,, (g_n)_{n \geqslant 1}\,\left\|\sum_{n \geqslant 1} \frac{g_n^{\omega}}{\lambda _n}e_n\right\|_{L^p} \leqslant R\right\}\,,
\]
which by independence has probability $\mathbb{P}(E)=\mathbb{P}\left(g_0>\frac{2\lambda_0 R}{\|e_0\|_{L^{p}}}\right)\mathbb{P}\left(\left\|\sum_{n \geqslant 1} \frac{g_n}{\lambda _n}e_n\right\|_{L^p}\leqslant R\right)$. Since $g_0$ is a Gaussian, $\mathbb{P}\left(g_0>\frac{2\lambda_0 R}{\|e_0\|_{L^{p}}}\right)$. Moreover, since $p>2$ is sufficiently close to $2$, and thanks to Lemma~\ref{3lemmaGainSobolev} we have 
\[
    \mathbb{P}\left(\left\|\sum_{n \geqslant 1} \frac{g_n}{\lambda _n}e_n\right\|_{L^p}\leqslant R\right) \geqslant 1-e^{-cR^2}>0\,.
\] 
Then we have $\mathbb{P}(E)>0$ and for all $\omega \in E$ there holds 
\[
    \left\|\sum_{n\geqslant 0}\frac{g_n^{\omega}}{\lambda_n}e_n\right\|_{L^p} \geqslant R\,.\qedhere
\]
\end{proof}

We remark that, if $\mathcal{N}_{\mathbb{C}}(0,\lambda^{-2}_n)$ denotes the law of a centred complex Gaussian of variance $\lambda_n^{-2}$, then we have
\[
    \mu = \bigotimes_{n \geqslant 0} \mathcal{N}_{\mathbb{C}}(0,\lambda_n^{-2}) = \mu_N \otimes \mu^{\perp}_N\,,
\]
where $\mu_N = \bigotimes_{0\leqslant n \leqslant N} \mathcal{N}_{\mathbb{C}}(0,\lambda_n^{-2})$, and $\mu_N^{\perp} = \bigotimes_{n>N} \mathcal{N}_{\mathbb{C}}(0,\lambda_n^{-2})$. 

Writing any $u\in X$ as $u=\displaystyle\sum_{n\geqslant 0} u_ne_n$, we see that the distribution of $\frac{g_n}{\lambda_n}$ is given by 
\[
    e^{-\frac{1}{2}\lambda_n^2|u_n|^2}\pi^{-1}\lambda_n^2\,\mathrm{d}u_n\,,
\] 
where $\mathrm{d}u_n$ is the Lebesgue measure on $\mathbb{C}$. Therefore we have
\[
    \mathrm{d}\mu_N(u_N) = \left(\prod_{0\leqslant n\leqslant N}\pi^{-1}\lambda_n^2\right) e^{-\frac{1}{2}\|u_N\|^2_{\mathcal{H}^1}}\,\mathrm{d}u_N\,,
\]
where $\mathrm{d}u_N=\prod_{0\leqslant n \leqslant N} \mathrm{d}u_n$, and $u_N=\sum_{0\leqslant n\leqslant N}u_ne_n\in X$. For this reason, we will write informally write that for any measurable set $A \subset X$,   
\[
    \mu(A) = \int_{A}e^{-\frac{1}{2}\|u\|^2_{\mathcal{H}^1}}\,\mathrm{d}u
\]

\subsection{Probabilistic estimates}

The main probabilistic estimate which will be used is a standard large deviation estimate in the context of the radial harmonic oscillator. 

\begin{lemma}\label{3lemmaGainSobolev} Let $d \geqslant 2$ and $\varepsilon >0$ arbitrarily small. Define \[s_r=\left\{
\begin{array}{ccc}
    d\left(\frac{1}{2}-\frac{1}{r}\right) & \text{if} & r\in \left(2,\frac{2d}{d-1}\right]\\
    1-d\left(\frac{1}{2}-\frac{1}{r}\right) & \text{if} & r\in \left(\frac{2d}{d-1},\frac{2d}{d-2}\right]
\end{array}
\right.\]
and $s_r^{-}:=s_r-\varepsilon$. Then:
\begin{enumerate}[label=(\textit{\roman*})]
    \item For $\mu$-almost every $u\in X$ one has $u \in \mathcal{W}^{s_r^{-},p}$.  
    \item More precisely there exists $C,c>0$ such that for every $p\geqslant 2$ and $\lambda >0$ we have 
    \[
        \mathcal{\mu} \left(u \in X, \; \|u\|_{\mathcal{W}^{s_r^{-},r}}>\lambda\right) \leqslant C e^{-c\lambda ^2}\,.
    \]
\end{enumerate}
\end{lemma}

\begin{proof} The proof is very similar to the proof of Lemma~3.3 in~\cite{burqThomannTzvetkov} with $N_0=0$ and $N=\infty$, and using the bounds~\ref{3lemmaGainLp}. See also~\cite{ayacheTzvetkov}. 
\end{proof}

The probabilistic smoothing gain is represented on Figure~\ref{3fig:probGain}. As we can see, there is always a good choice of $L^r$ scale to gain almost $\frac{1}{2}$ derivatives. 
\begin{figure}
    \centering
    \includegraphics[width=0.9\textwidth]{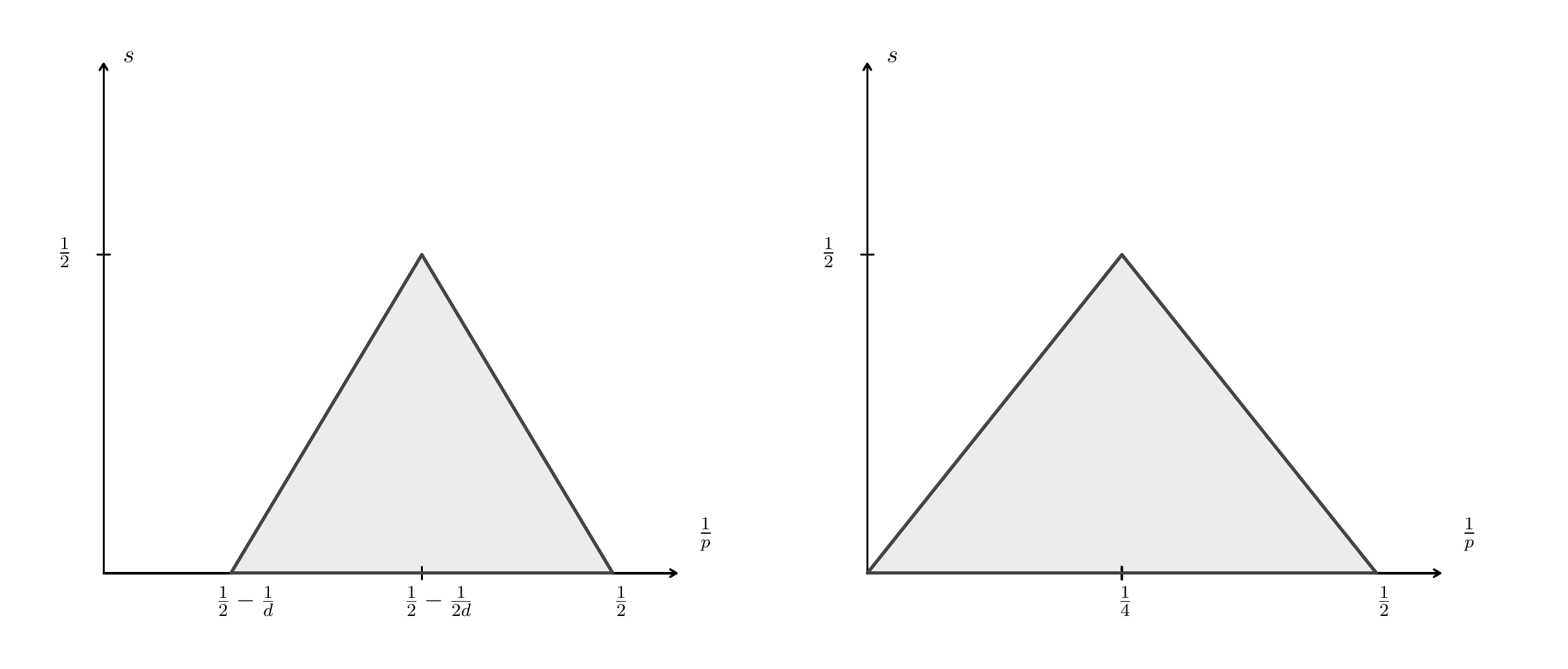}
    \caption{Probabilistic gain on the grey zone. The picture on the right describes the case $d=2$ and the picture on the left is the case $d\geqslant 3$, where we can see that the probabilistic gain holds for sufficiently large $\frac{1}{p}$.}
    \label{3fig:probGain}
\end{figure}

Lemma~\ref{3lemmaGainSobolev} immediately implies the following consequence.

\begin{corollary}\label{3coroProba} There exists $C,c>0$ such that for every $q \geqslant 1$, $r \in \left[2,\frac{2d}{d-2}\right]$ and $\sigma \in [0,s_r^{-})$ one has
\[
    \mu\left(u\in X, \|e^{-itH}u\|_{L^q_{[-\pi,\pi]}\mathcal{W}^{\sigma, r}} >\lambda \right) \leqslant Ce^{-c\lambda^2}\text{ for all } \lambda \geqslant 1\,.
\]
\end{corollary}

For details see~\cite{thomann}. 

\subsection{Deterministic estimates}

We will need dispersion and Strichartz estimates in order to construct local solutions. 
A pair $(q,r) \in [2,\infty]^2$ is Schrödinger admissible in dimension $d \geqslant 1$ if $(q,r,d)$ satisfies \[\frac{2}{q}+\frac{d}{r}=\frac{d}{2}\; \text{ and }\; (q,r,d) \neq (2,\infty,2)\,.\] 

\begin{remark} We should check on many occasions that a given pair $(q,r)$ is admissible. One should keep in mind that the condition $q,r \geqslant 2$ is required. Taking into account the relation between $q$ and $r$, the requirement $q, r \geqslant 2$ is equivalent to $q \geqslant 2$. In the same spirit, $(q', r')$ is such that $(q,r)$ is admissible if and only if $\frac{1}{2} \leqslant \frac{1}{r'} \leqslant \frac{1}{2}+\frac{1}{d}$. 
\end{remark}

We define the Strichartz space 
\[
    Y^{\sigma}_{[t_0,t]}\coloneqq \bigcap_{(q,r) \text{ admissible }} L^q([t_0,t],\mathcal{W}^{\sigma,r})
\] 
and its dual $\tilde{Y}^{\sigma}_{[t_0,t]}$ that we respectively endow with the norms:
\[
    \|u\|_{Y^{\sigma}_{[t_0,t]}} \coloneqq \sup_{(q,r) \text{ admissible}} \|u\|_{L^q([t_0,t],\mathcal{W}^{\sigma,r})} \text{ and } \|u\|_{\tilde{Y}^{\sigma}_{[t_0,t]}} \coloneqq \inf_{(q,r)\; \text{ admissible}}\; \|u\|_{L^{q'}([t_0,t],\mathcal{W}^{\sigma,r'})}\,.
\]
We also write $Y^{\sigma}_{t_0,\tau}$ for $Y^{\sigma}_{[t_0,t_0+\tau]}$. 

Let $\sigma>0$. Using the Sobolev embedding, if $(q,\tilde{r})$ is Schrödinger admissible and $r\geqslant 1$ is such that $\frac{d}{\tilde{r}} = \frac{d}{r}+\sigma$ then by the Sobolev embedding,
\[
    \|u\|_{L^q_TL^r} \lesssim \|u\|_{L^q_T\mathcal{W}^{\sigma,\tilde{r}}} \lesssim \|u\|_{Y^{\sigma}_T}\,.
\]
Thus we often refer to $(q,r)$ being $\sigma$-Schrödinger admissible if $\frac{2}{q}+\frac{d}{r}=\frac{d}{2}-\sigma$, $q\geqslant 2$ and $r\geqslant 1$.  

\begin{proposition}[Dispersion and Strichartz estimates]\label{3propStrichartz} Given $\sigma \geqslant 0$ and $t_0,T \in (0,\frac{\pi}{4})$, the following estimates hold.
\begin{enumerate}[label=(\textit{\roman*})]
    \item (Dispersion) For $r \geqslant 2$ and $t > 0$ holds $\|e^{itH}\|_{L^{r'}\to L^r} \leqslant C(d,r)|t|^{-d\left(\frac{1}{2}-\frac{1}{r}\right)}$; 
    \item (Homogeneous) $\|e^{itH}\|_{\mathcal{H}^{\sigma}\to Y_{[t_0,T]}^{\sigma}} \leqslant C(d)$;
    \item (Non-homogeneous) $\displaystyle \left\| \int_0^t e^{i(t-s)H}\,\mathrm{d}s \right\| _{\tilde{Y}^{\sigma}_T \to Y^{\sigma}_T} \leqslant C(d)$. 
\end{enumerate}
\end{proposition}

\begin{proof} The proof is carried out in detail in dimension two in~\cite{deng}, Proposition~2.7. For dimension $d \geqslant 3$ we refer to~\cite{poiret}. In both cases the proof is an application of the $TT^*$ argument and the Christ-Kiselev lemma along with a dispersive estimate for $e^{-itH} : L^1 \to L^{\infty}$ which is obtained through stationary phase estimates.
\end{proof}

\subsection{Spaces of functions} 

In section~\ref{3secLocal} we will not construct local solutions to~\eqref{3NLSH} for any initial data in $X$, but rather for initial data in a subspace $X^{\sigma}$, but such that $\mu (X^{\sigma})=1$. 

Let $d \geqslant 2$, $p \in (1,1+\frac{4}{d})$, then depending on whether $p \leqslant 1+ \frac{3}{d-2}$ or not we introduce the following spaces. 

\begin{itemize}
    \item \textit{Case 1.} If $p\leqslant 1+\frac{3}{d-2}$ then we recall that $\sigma (p,d)=\frac{1}{2}$ and define, for any $\sigma \in (0,\sigma (p,d))$ a complete space $X^{\sigma}$ by the norm 
    \[
        \|u\|_{X^{\sigma}} \coloneqq \|e^{-itH}u\|_{L^{q_2}_{[-\pi,\pi]}\mathcal{W}^{\sigma ,r_2}} + \|e^{-itH}u\|_{L^{a(p-1)}_{[-\pi,\pi]}L^{b(p-1)}}\,,
    \]
    where $(q_2,r_2)=(4,\frac{2d}{d-1})$ is Schrödinger admissible, and $a, b$ only depend on $p, d, \sigma$ and will be defined in the course of the proof of Lemma~\ref{3lemmaContraction} in a way such that $(a(p-1),b(p-1))$ is $\sigma$-Schrödinger admissible and $b(p-1) \leqslant \frac{2d}{d-2}$. 
    \item \textit{Case 2.} If $p>1+\frac{3}{d-2}$ then we recall that $\sigma (p,d)=2-\frac{d-2}{2}(p-1)$ and for any $\sigma \in (0,\sigma (p,d))$ we define a complete space $X^{\sigma}$ by the norm 
    \[
        \|u\|_{X^{\sigma}} \coloneqq \|e^{-itH}u\|_{L^{q_2}_{[-\pi,\pi]}\mathcal{W}^{\sigma ,r_2}} + \|e^{-itH}u\|_{L^{a(p-1)}_{[-\pi,\pi]}L^{b(p-1)}}\,,
    \]
    where $(q_2,r_2)=(\frac{4}{(d-2)(p-1)-2},\frac{2d}{d+2-(d-2)(p-1)})$ is again a Schrödinger admissible pair and $a, b$ only depend on $p$ and $d$ and will be chosen in Lemma~\ref{3lemmaContraction} in order for $(a(p-1),b(p-1))$ to be $\sigma$-Schrödinger admissible and also satisfy $b(p-1)\leqslant\frac{2d}{d-2}$.
\end{itemize}

As a consequence of Corollary~\ref{3coroProba} we have the following.  

\begin{corollary}\label{coro:XsigmaProbaEstimates} Let $d\geqslant 2$, $p\in[1,1+\frac{4}{d})$ and $\sigma \in (0,\sigma(p,d))$, then 
\[
    \mu(\{u \in X, \|u\|_{X^{\sigma}}>\lambda\}) \leqslant e^{-c\lambda ^2}\,,
\]
and in particular $\mu(X^{\sigma})=1$. 
\end{corollary}

Next, we establish some useful properties of the spaces $X^{\sigma}$. 

\begin{lemma}\label{lem.aa} Let $\sigma \in (0,\sigma(p,d))$, then the following properties hold: 
\begin{enumerate}[label=(\textit{\roman*})]
    \item For any $u\in X^{\sigma}$, $\|e^{-itH}u\|_{X^{\sigma}}=\|u\|_{X^{\sigma}}$
    \item There is a continuous embedding $\mathcal{H}^{\sigma} \hookrightarrow X^{\sigma}$. As a consequence, the embedding $Y^{\sigma}_{[t_0,t_1]} \hookrightarrow X^{\sigma}$ is continuous. 
    \item $\mathbb{S}_N$ acts continuously on $X^{\sigma}$ and $Y^{\sigma}_{t_0,\tau}$ with continuity constant independent of $N$. As a consequence, if $u\in X^{\sigma}$ then $\mathbf{S}_N u \to u$ in $X^{\sigma}$. 
    \item If $\sigma ' > \sigma$ then $X^{\sigma'} \hookrightarrow X^{\sigma}$ is compact and more precisely there holds 
    \[
        \|(\operatorname{id}-\mathbf{S}_N)u\|_{X^{\sigma}} \lesssim N^{-(\sigma'-\sigma)}\|u\|_{X^{\sigma '}}\,. 
    \]
\end{enumerate}
\end{lemma}

\begin{proof} 
\begin{enumerate}[label=(\textit{\roman*})]
    \item This comes from the $2\pi$-periodicity of $e^{itH}$, which is the case because the eigenvalues of $H$, the $\lambda_n^2$ are integers. 
    \item Since in the definition of $X^{\sigma}$, $(q_2,r_2)$ is admissible and $(a(p-1),b(p-1))$ is $\sigma$-admissible, this follows from Strichartz estimates. 
    \item These uniform continuity properties boil down to the bound $\|\mathbf{S}_N\|_{L^r \to L^r} \lesssim_r 1$, available for any $r\in[1,\infty)$, see~\cite{burqThomannTzvetkov}, Proposition~4.1 for example.
    \item The inequality is a consequence of Bernstein estimates and the compactness follows from the inequality and the fact that $H$ has a spectrum made of countably many eigenfunctions. \qedhere
\end{enumerate}
\end{proof}

\section{The probabilistic local Cauchy theory}\label{3secLocal}

In this section we construct a local Cauchy theory for~\eqref{3NLSH} and prove quantitative approximation results of solutions to~\eqref{3NLSH} by solutions of
\begin{equation}
\label{HNLSN}\tag{HNLS${}_N$}
    \left\{
    \begin{array}{c}
         i\partial_t v-Hv=\cos(2t)^{-\alpha(p,d)} \mathbf{S}_N(|\mathbf{S}_Nv|^{p-1}\mathbf{S}_Nv)  \\
         v(0)=u_0\,.
    \end{array}
    \right.
\end{equation}

\subsection{Local construction of solutions}

The main result of this section is a flexible local well-posedness result, where the initial data is given at a time $t_0$. We state only a forward in time result in $[0,\frac{\pi}{4})$, a similar statement holds for negative times.

\begin{proposition}[Local Cauchy theory]\label{3propLocal} Let $d \geqslant 2$, $t_0 \in (0,\frac{\pi}{4})$, $p \in \left[1,1+\frac{4}{d}\right)$, $\sigma \in [0,\sigma(p,d))$ (defined by~\eqref{3eqSigma}) and $\lambda>0$. There exists $\tau=\tau (t_0,\lambda) >0$ of the form:
\begin{equation}
    \label{eq:tauDef}
    \tau \sim \lambda^{-L}\left(\frac{\pi}{4}-t_0\right)^{K}\,,  
\end{equation}
for some constants $L=L(p,d,\sigma) >0$, $K = K(p,d,\sigma)\geqslant 1$ such that the following statements hold.

\bigskip 

\noindent \emph{(Existence, Uniqueness)} For any $u_0 \in B_{X^{\sigma}}(\lambda)$ there exists a unique solution 
\[
    v \in e^{-i(t-t_0)H}u_0 + Y_{t_0,\tau}^{\sigma}
\] 
to~\eqref{3NLSH} on $[t_0,t_0+\tau]$ such that $v(t_0)=u_0$, where uniqueness hold for $w=v-e^{-i(t-t_0)H}u_0$ in $Y_{[t_0,t_0+\tau]}^{\sigma}$. Moreover, $w \in \mathcal{C}^0([t_0,t_0+\tau],\mathcal{H}^\sigma)$.

\bigskip 

\noindent \emph{(Continuity of the flow)} $\phi_t(u_0) \coloneqq v(t)$ defines a flow for $t\in[t_0,t_0+\tau]$ acting on $B_{X^{\sigma}}(\lambda)$ which is such that for any $u_0$, $\tilde{u}_0 \in B_{X^{\sigma}}(\lambda)$ there holds
\begin{equation}
    \label{eq:continuityXsigmaPrime}
    \|\phi_t(u_0) - \phi_t(\tilde{u}_0)\|_{X^{\sigma'}} \lesssim \|u_0-\tilde{u}_0\|_{X^{\sigma'}}\,,
\end{equation}
for any $\sigma'\in[0,\sigma)$. 

\bigskip 

\noindent \emph{(Persistence of regularity)} There exists a constant $C>0$ such that the following holds. If we assume furthermore that $u_0 \in B_{X^{\sigma '}}(R)$ for some $\sigma ' \in (\sigma, \sigma (p,d))$ and $R>0$, then the associated constructed solutions $v = v_L + w$ satisfies for any $t\in[t_0,t_0+\tau]$, 
\begin{equation}
    \label{eq:persistence}
    \|v(t)\|_{X^{\sigma '}}\leqslant CR\,,
\end{equation}
and
\begin{equation}
    \label{eq:persistenceBis} 
    \|w\|_{Y^{\sigma'}_{t_0,\tau}}\leqslant CR\,,
\end{equation}
\end{proposition}

\begin{remark}\label{rem:localTheory} Since $\mathbf{S}_N$ is continuous $L^q \to L^q$ for any $q\in[2,\infty)$ the above local Cauchy theory applies verbatim to~\eqref{HNLSN}, uniformly in $N$ with the same $\tau$ given by~\eqref{3propLocal}.   
\end{remark}

\begin{remark} It is important in~\eqref{eq:persistence} that the local well-posedness time $\tau$ does not depend on $R$ but only on $\lambda$, that is, the size of $u_0$ in $X^{\sigma}$. 
\end{remark}

\begin{remark} 
Observe that for any $t_0 \in [0,t_1]$, and up to reducing the local well-posedness time, $\tau = \tau (t_0)$ is larger or equal to that obtained by taking $t_0=t_1$. When $\tau$ is chosen this way, we say that $\tau$ is chosen uniformly in $[0,t_1]$. 
\end{remark}

The strategy for proving local well-posedness is a fixed-point argument, which turns out to be easier than the one in~\cite{BT19}. In the latter case, the lack of sufficient regularisation in the the stochastic linear term requires a more refined analysis, whereas in our case, Lemma~\ref{3lemmaGainSobolev} is powerful enough for our purpose. We write $\Phi : Y^{\sigma}_{t_0,\tau} \longrightarrow \mathcal{S}'$ for the map defined by
\begin{equation}
\label{3defPhi}
     \Phi (w)(t) := -i \int_{t_0}^{t} e^{-i(t-s)H} \left(\cos (2s)^{-\alpha(p,d)}|v_L(s)+w(s)|^{p-1}(v_L(s)+w(s))\right)\mathrm{d}s \,,
\end{equation}
where $v_L(t):=e^{-i(t-t_0)H}u_0$ is the linear evolution of the initial data $u_0$. 
The crucial estimates needed to run the fixed point argument are gathered in the following lemma. 

\begin{lemma}\label{3lemmaContraction} Let $\Phi$ defined by~\eqref{3defPhi}. Let $d\geqslant 2$. Let also $p \in \left[1,1+\frac{4}{d}\right)$, $\sigma \in [0,\sigma(p,d))$, $t_0\in[0,\frac{\pi}{4})$ and $\tau \leqslant \frac{1}{2}\left(\frac{\pi}{4}-t_0\right)$. There exists $\kappa >0$ which only depends on $p$, $d$ and $\sigma$, such that for $u_0\in B_{X^{\sigma}}(\lambda)$ and for $w_1$, $w_2$, $w$ in the ball $B_{Y^{\sigma}_{t_0,\tau}}(\lambda)$, one has:
\begin{equation}
    \label{3eqStab}
    \|\Phi (w)\|_{Y^{\sigma}_{t_0,\tau}} \lesssim_{d,p} \left(\frac{\pi}{4}-t_0\right)^{-\alpha(p,d)}\tau^{\kappa}\lambda^p\,,
\end{equation}
and
\begin{equation}
    \label{3eqContract}
    \|\Phi (w_1)-\Phi (w_2)\|_{Y^{0}_{t_0,\tau}} \lesssim_{d,p} \left(\frac{\pi}{4}-t_0\right)^{-\alpha(p,d)}\tau^{\kappa}\lambda^{p-1}\|w_1-w_2\|_{Y^{0}_{t_0,\tau}}\,.
\end{equation}
Moreover, for $u_0, \tilde{u}_0 \in B_{X^{\sigma}}(\lambda)$, let us denote by $\Phi, \tilde{\Phi}$ the associated maps defined as in~\eqref{3defPhi}. Then, for any $w, \tilde{w} \in B_{Y^{\sigma}_{t_0,\tau}}(\lambda)$ there holds: 
\begin{equation}
    \label{3eqContinuity}
    \|\Phi (w)-\tilde{\Phi}(\tilde{w})\|_{Y^0_{t_0,\tau}} \lesssim_{d,p} \left(\frac{\pi}{4}-t_0\right)^{-\alpha(p,d)}\tau^{\kappa}\lambda^{p-1}\left(\|u_0-\tilde{u}_0\|_{X^0} + \|w-\tilde{w}\|_{Y^{0}_{t_0,\tau}}\right)\,.
\end{equation}
\end{lemma}

\begin{proof}[Proof of Lemma~\ref{3lemmaContraction}] \textit{Proof of~\eqref{3eqStab}}. Let $u_0 \in B_{X^{\sigma}}(\lambda)$. Recall that we need to fix the parameters $a, b$ of the space $X^{\sigma}$, which we will do in this proof. From~\eqref{3defPhi} and by the non-homogeneous Strichartz estimates with an admissible pair~$(q,r)$, one has
\begin{align*}
    \|\Phi (v)\|_{Y^{\sigma}_{[t_0,t_0+\tau]}} & \leqslant \|\cos (2s)^{-\alpha(p,d)}|v_L(s)+w(s)|^{p-1}(v_L(s)+w(s))\|_{L^{q'}_{[t_0,t_0+\tau]}\mathcal{W}^{\sigma,r'}} \\
    & \lesssim \left( \int_{t_0}^{t_0+\tau} \left(\frac{\pi}{4}-s\right)^{-\alpha(p,d)q_1} \right)^{\frac{1}{q_1}} \||v_L+w|^{p-1}(v_L+w)\|_{L^{\tilde{q}}_{[t_0,t_0+\tau]}\mathcal{W}^{\sigma,r'}}\\
    & \lesssim \left(\frac{\pi}{4}-t_0\right)^{-\alpha(p,d)}\tau^{\frac{1}{q_1}}\||v_L+w|^{p-1}(v_L+w)\|_{L^{\tilde{q}}_{[t_0,t_0+\tau]}\mathcal{W}^{\sigma,r'}}\,,
\end{align*}
where we used Hölder's inequality with $q_1, \tilde{q}$ to be adjusted and such that $\frac{1}{q_1}+\frac{1}{\tilde{q}}=\frac{1}{q'}$. We also used that $|\cos(2s)| \gtrsim \left(\frac{\pi}{4}-s\right)$ and $\tau \leqslant \frac{1}{2}\left(\frac{\pi}{4}-t_0\right)$. 

Then by the Sobolev product estimates from Lemma~\ref{3technicalTools} and the triangle inequality one can write
\begin{align*}
    \|\Phi (w)\|_{Y^{\sigma}_{[t_0,t_0+\tau]}} &\lesssim \left(\frac{\pi}{4}-t_0\right)^{-\alpha(p,d)}\tau^{\frac{1}{q_1}}  \|v_L+w\|_{L^{q_2}\mathcal{W}^{\sigma,r_2}}\||v_L+w|^{p-1}\|_{L^{q_3}L^{r_3}} \\
    &\lesssim \left(\frac{\pi}{4}-t_0\right)^{-\alpha(p,d)}\tau^{\frac{1}{q_1}} \left(\|v_L\|_{L^{q_2}\mathcal{W}^{\sigma,r_2}} + \|w\|_{L^{q_2}\mathcal{W}^{\sigma,\frac{2d}{d-1}}} \right) \\
    & \times \left(\|v_L\|_{L^{q_3(p-1)}L^{r_3(p-1)}}^{p-1}+\|w\|_{L^{q_3(p-1)}L^{r_3(p-1)}}^{p-1}\right)\,,
\end{align*}
where the parameters $q_1,q_2,q_3$ and $r_3$ need to satisfy \[\frac{1}{q_1}+\frac{1}{q_2}+\frac{1}{q_3}=\frac{1}{q'}\,,\]
and
\[\frac{1}{r_2}+\frac{1}{r_3}=\frac{1}{r'}\,\cdotp\]
We recall that $q',r'$ need to satisfy $q',r' \in [1,2]$ and 
\[\frac{2}{q'}+\frac{d}{r'}=\frac{d}{2}+2\,.\] 

Assume that we are able to choose the parameters $q_1, q_2, q_3, q', r_2, r_3, r'$ so that the norm $L^{q_2}L^{r_2}$ is Schrödinger admissible,  $L^{q_3(p-1)}L^{r_3(p-1)}$ is $\sigma$-Schrödinger admissible, then we will meet the conditions in the construction of the $X^{\sigma}$ norm and since $u_0 \in B_{X^{\sigma}}(\lambda)$ and $w\in B_{Y^{\sigma}_{t_0,\tau}}(\lambda)$ we will obtain 
\[
    \|\Phi (w)\|_{Y^{\sigma}_{[t_0,t_0+\tau]}} \lesssim \left(\frac{\pi}{4}-t_0\right)^{-\alpha(p,d)}\tau^{\frac{1}{q_1}} \lambda^p\,,
\] 
hence~\eqref{3eqStab}. 

In order to explain how to choose the parameters, we distinguish several cases. 

\bigskip

\noindent\textit{Case 1.} We assume $1<p\leqslant\text{min} \{1+ \frac{3}{d-2},1+\frac{4}{d}\}$ so that $\sigma(p,d)=\frac{1}{2}$, and thus $0<\sigma < \frac{1}{2}$. To control the $\mathcal{W}^{\sigma , r_2}$ norm of $v_L$ we take $r_2=\frac{2d}{d-1}$ 
 
We then choose $q_2=4$, so that $(q_2,r_2)$ is Schrödinger admissible, hence
\[
    \|w\|_{L^{4}\mathcal{W}^{\sigma,\frac{2d}{d-1}}} \leqslant \|w\|_{Y^{\sigma}_{[t_0,t_0+\tau]}} \leqslant \lambda\,.
\]
We also need that $r_3(p-1) \leqslant \frac{2d}{d-2}$, which is implied by $r'$ satisfying $\frac{1}{r'} \geqslant p\left(\frac{1}{2}-\frac{1}{d}\right)+\frac{1}{2d}$. Note that in order to ensure that $q' \in [1,2]$ we need $\frac{1}{2} \leqslant \frac{1}{r'} \leqslant \frac{1}{2}+\frac{1}{d}$. The above restriction on $r'$ makes it possible to choose such an $r'$ if and only if \[
    p\left(\frac{1}{2}-\frac{1}{d}\right)+\frac{1}{2d}\leqslant \frac{1}{2}+\frac{1}{d}\,,
\]
which gives the condition $p \leqslant 1+ \frac{3}{d-2}$, satisfied by hypothesis. Then we can choose $\frac{1}{r'}=p\left(\frac{1}{2}-\frac{1}{d}\right)+\frac{1}{2d}$ (which fixes $r_3$ as desired, and satisfies $r' \in [1,2]$), and an associate $q'$, which is also in the range $[1,2]$. To make sure that we can control the term $\|w\|_{L^{q_3(p-1)}L^{r_3(p-1)}}$ by the Strichartz norm $\|w\|_{Y^{\sigma}_{t_0,\tau}}$, we use the Sobolev embedding (which is applicable since $\frac{1}{r_3(p-1)}=\frac{2d}{d-2} \leqslant \frac{1}{2}\leqslant 1$), which reads 
\[
    \|v_L\|_{L^{q_3(p-1)}L^{r_3(p-1)}} \lesssim \|v_L\|_{L^{q_3(p-1)}W^{\sigma,\tilde{r}_3(p-1)}}\,
\]
where $\tilde{r}_3$ satisfies the Sobolev condition $\frac{1}{\tilde{r}_3(p-1)} \leqslant \frac{\sigma}{d} + \frac{1}{r_3(p-1)}$. If moreover $(q_3(p-1),\tilde{r}_3(p-1))$ is Schrödinger admissible, that is $\frac{2}{q_3(p-1)}+\frac{d}{\tilde{r}_3(p-1)}=\frac{d}{2}$ and $q_3(p-1), \tilde{r}_3(p-1) \geqslant 2$, then we can control $\|w\|_{L^{q_3(p-1)}L^{r_3(p-1)}}$ by the Strichartz norm $\|w\|_{Y^{\sigma}_{t_0,\tau}}$. We remark that the Sobolev condition can be written in the form 
\[
    \frac{1}{r_3(p-1)} \geqslant -\frac{\sigma}{d}+\frac{1}{2} - \frac{2}{dq_3(p-1)}\,,
\] 
which can be written in variables $q_1,p,d$ as 
\[
    \frac{1}{q_1} \leqslant 1-\frac{(p-1)(d-2\sigma)}{4}\,,
\] 
which is always satisfied for large $q_1$, since the quantity on the right-hand side is always bounded from below by $1-\frac{\min\{\frac{3}{d-2},\frac{4}{d}\}}{4}(d-2\sigma)>0$. Now we can choose freely $q_3$ such that $q_3(p-1)\geqslant 2$ large enough so that $\frac{1}{q_1}+\frac{1}{q_3}=\frac{1}{q'}-\frac{1}{4}$, which is possible since $\frac{1}{q'}>\frac{1}{4}$.

\bigskip

\noindent\textit{Case 2.} We assume $d \geqslant 8$, $p \in (1+\frac{3}{d-2},1+\frac{4}{d})$ and thus $\sigma(p,d)=2-\frac{d-2}{2}(p-1)$. We set
\[
    q_2\coloneqq\frac{4}{(d-2)(p-1)-2} \geqslant 2 \text{ and } r_2\coloneqq\frac{2d}{d+2-(d-2)(p-1)} \geqslant 2\,\cdotp
\] 
Because $(q_2,r_2)$ is Schrödinger admissible we have $\|w\|_{L^{q_2}\mathcal{W}^{\sigma,r_2}} \leqslant \lambda$.

As above we need to make sure we can find $r_3$ such that $2\leqslant r_3(p-1) \leqslant \frac{2d}{d-2}$ and 
\[
    \frac{1}{r_3}+\frac{d+2-(d-2)(p-1)}{2d}=\frac{1}{r'}\,\cdotp
\] 
The condition that $1 \leqslant q' \leqslant 2$ translates into $\frac{1}{2} \leqslant \frac{1}{r'} \leqslant \frac{1}{2}+\frac{1}{d}$. The condition on $r_3$ is written: 
\[
    \frac{1}{2}+\frac{1}{d}+\frac{p-1}{d} \geqslant \frac{1}{r'} \geqslant \frac{(p-1)(d-2)}{2d}+\frac{d+2-(d-2)(p-1)}{2d}=\frac{1}{2}+\frac{1}{d}\,\cdotp
\]
hence we choose $\frac{1}{r'}=\frac{1}{2}+\frac{1}{d}$, satisfying both conditions, thus $q'=2$. We now have $\frac{1}{r_3}=\frac{(d-2)(p-1)}{2d}$, so that $r_3(p-1) \geqslant 2$. 

To finish the proof we need to control the norm $\|w\|_{L^{q_3(p-1)}L^{r_3(p-1)}}$. Arguing as in \textit{Case 1.} we derive a condition on $q_1$, which is
\begin{equation}
    \label{eq:condQ1}
    \frac{1}{q_1} \leqslant 1-\left(\frac{d}{4}-\frac{\sigma}{2}\right)(p-1)\,,
\end{equation}
which is satisfied for large $q_1$ as soon as 
\[
    1-\left(\frac{d}{4}-\frac{\sigma(d)}{2}\right)(p-1) >0\,,
\] 
that is $p-1 \leqslant \frac{-d+4+\sqrt{d^2+8d-16}}{2(d-2)}$. This is satisfied since $p-1 \leqslant \frac{4}{d}\leqslant \frac{-d+4+\sqrt{d^2+8d-16}}{2(d-2)}$. Then we choose $q_3$ large enough so that $q_3(p-1)\geqslant 2$ and we conclude.

\bigskip

\noindent\textit{Proof of~\eqref{3eqContract}.} We use the following, valid for any complex numbers $a, b, c$: 
\[
    ||a+b|^{p-1}(a+b)-|a+c|^{p-1}(a+c)| \lesssim_p |b-c|(|a|^{p-1}+|b|^{p-1}+|c|^{p-1})\,,
\]
which implies that for any Schrödinger admissible pair $(q,r)$, there holds 
\[
    \|\Phi(w_1)-\Phi(w_2)\|_{Y^0_{t_0,\tau}} \lesssim_p \|\cos(2t)^{-\alpha(p,d)}|w_1-w_2|(|v_L|^{p-1}+|w_1|^{p-1}+|w_2|^{p-1})\|_{L^{q'}_{[t_0,t_0+\tau]}L^{r'}}\,. 
\]  
With the use of the Hölder inequality we obtain
\begin{multline*}
       \|\Phi(w_1)-\Phi(w_2)\|_{Y^0_{t_0,\tau}} \lesssim_p  \left(\frac{\pi}{4}-t_0\right)^{-\alpha (p,d)}\tau ^{\frac{1}{q_1}}\|w_1-w_2\|_{L^{q_2}L^{r_2}} \\
    \times \left(\|v_L\|^{p-1}_{L^{q_3(p-1)}L^{r_3(p-1)}}+ \|w_1\|^{p-1}_{L^{q_3(p-1)}L^{r_3(p-1)}} +\|w_2\|^{p-1}_{L^{q_3(p-1)}L^{r_3(p-1)}}\right)\,, 
\end{multline*}
where $q_1$, $q_2$, $q_3$, $r_2$, $r_3$, $q'$, $r'$ satisfy: 
\[
    \frac{1}{q_1}+\frac{1}{q_2}+\frac{1}{q_3}=\frac{1}{q'}    
\]
and 
\[
    \frac{1}{r_2}+\frac{1}{r_3}=\frac{1}{r'}\,,
\]
which are the same conditions as in the proof of~\eqref{3eqStab} thus with the same choices of $q_1$, $q_2$, $q_3$, $r_2$, $r_3$, $q'$, $r'$ we obtain~\eqref{3eqContract}.

\bigskip

\noindent \textit{Proof of~\eqref{3eqContinuity}.} The proof is similar to that of~\eqref{3eqContract}, we omit the details. 
\end{proof}

We give the details of how the estimates of~Lemma~\ref{3lemmaContraction} imply Proposition~\ref{3propLocal}. Because we deal with some $1<p<2$, the fixed-point procedure seems not possible to work directly in the space $Y^{\sigma}_{t_0,\tau}$, therefore we prove contraction in a weaker norm. 

\begin{proof}[Proof of Proposition~\ref{3propLocal}] From Lemma~\ref{3lemmaContraction} we see that $\Phi$ stabilises the ball $B_{Y^{\sigma}_{t_0,\tau}}(\lambda)$ as soon as $\tau \leqslant C \left(\frac{\pi}{4}-t_0\right)^{\alpha(p,d)/\kappa}\lambda^{-(p-1)/\kappa}$, for some constant $C>0$. 

Moreover, from~\eqref{3eqContract} we see that on the ball $B_{Y^{\sigma}_{t_0,\tau}}(\lambda)$ (up to reducing the constant $C$), $\Phi$ is a contraction in the norm $Y^0$, thus has a unique fixed-point by the Banach Fixed Point Theorem. 

To see that $w \in \mathcal{C}^0([t_0, t_0+\tau],\mathcal{H}^{\sigma})$, we let $t_1<t_2 \in [t_0, t_0+\tau]$ and write
\[
    w(t_1)-w(t_2)=\int_{t_1}^{t_2}\cos(2t')^{-\alpha (p,d)}\left(|v_L(t')+w(t')|^{p-1}(v_L(t')+w(t'))\right)\,\mathrm{d}t'\,,
\]
thus by the triangle inequality and~\eqref{3eqStab} we have
\[
    \|w(t_1)-w(t_2)\|_{\mathcal{H}^{\sigma}} \leqslant_{p, d} \left(\frac{\pi}{4}-t_0\right)^{-\alpha (p,d)}|t_1-t_2|^{\kappa}\lambda^p\,,
\]
which implies the claimed continuity in time. 

Continuity of the flow in the $X^0$ topology follows from~\eqref{3eqContinuity}: to this end, let $u_0, \tilde{u}_0 \in B_{X^{\sigma}}(\lambda)$. From the local theory, we can construct associated solutions $v=v_L+w$, and $\tilde{v}=\tilde{v}_L+\tilde{w}$ on $[t_0,t_0+\tau]$ such that, thanks to~\eqref{3eqContinuity} and the definition of $\tau$, we have: 
\[
    \|w-\tilde{w}\|_{Y^0_{t_0,\tau}} = \|\Phi(w)-\tilde{\Phi}(\tilde{w})\|_{Y^0_{t_0,\tau}} \lesssim_{p,d} \frac{1}{2}\left(\|u_0-\tilde{u}_0\|_{X^0} + \|w-\tilde{w}\|_{Y^0_{t_0,\tau}}\right)\,,
\]
which yields $\|w-\tilde{w}\|_{Y^0_{t_0,\tau}} \lesssim \|u_0-\tilde{u}_0\|_{X^0}$. Therefore, by the triangle inequality we obtain and Lemma~\ref{lem.aa} we obtain
\[
    \|v(t)-\tilde{v}(t)\|_{X^0} \lesssim \|u_0-\tilde{u}_0\|_{X^0}\,,
\]
which proves the continuity in the $X^0$ norm. Interpolating with the fact that $\phi_t$ is bounded in the $X^{\sigma}$ norm (which is a consequence of~\eqref{3eqStab}) yields the continuity for all $\sigma ' <\sigma$. 

Finally, persistence of regularity~\eqref{eq:persistence} follows from the stabilisation estimates of Lemma~\eqref{3lemmaContraction} with $\sigma '$ replaced by $\sigma$. (For details, see for example the proof of (8.7) in~\cite{BT19}).  
\end{proof}

\subsection{Uniform approximation estimates} 

The construction of a global flow for~\eqref{3NLSH} relies on a perturbation lemma. 

\begin{lemma}[Local approximation] \label{3approxLemma} 
Let $d\geqslant 2$, $p\in \left(1,1+\frac{4}{d}\right)$, $\sigma \in (0,\sigma (p,d))$ and $t_0 \in (-\frac{\pi}{4},\frac{\pi}{4})$. There exists $\tau =\tau(d,p,\sigma,t_0) >0$ of the form~\eqref{eq:tauDef} such that:  
\begin{enumerate}[label=(\textit{\roman*})]
    \item For any $u_0\in B_{X^{\sigma}}(\lambda)$, there exists a unique solution $v=v_L+w$ (resp. $v_N=v_L+w_N$) in $e^{-i(t-t_0)}u_0 + Y^{\sigma}_{t_0,\tau}$ to~\eqref{3NLSH} (resp.~\eqref{HNLSN}).
    \item For any $u_0, u_0^{(N)} \in B_{X^{\sigma}}(\lambda)$ such that $\|u_{0}^{(N)}-u_0\|_{X^{\sigma '}} \to 0$ for all $\sigma ' < \sigma$, we have that for all $t\in[t_0,t_0+\tau]$, and all $\sigma ' \in (0,\sigma)$, there holds
    \[
        \|v(t)-v_N(t)\|_{X^{\sigma'}} \to 0\,. 
    \]
    \item Assume that $u_0^{(N)}, u_0 \in B_{X^{\sigma '}}(R)\cap B_{X^{\sigma}}(\lambda)$ for some $\sigma ' \in (\sigma, \sigma (p,d))$, are such that for all $0\leqslant \delta < \sigma'$, 
    \[
        \|u_0-u_0^{(N)}\|_{X^{\delta}}\leqslant C_1N^{\delta - \sigma '}\,,
    \]
    where $C_1$ does not depend on $u_0$.   
    Then there exists a constant $C=C(t_0,\tau,R,\lambda)>0$ such that for every $\delta<\sigma'$, there holds 
    \[
        \|w-w_N\|_{Y^{\delta}_{t_0,\tau}} \leqslant CN^{\delta-\sigma'}\,,
    \]
    for $N$ sufficiently large. 
\end{enumerate}
\end{lemma}

\begin{proof} (\textit{i}) directly follows from Proposition~\ref{3propLocal} and Remark~\ref{rem:localTheory}, and provide solutions which we write $v(t)=e^{-itH}u_0 + w(t)=v_L(t)+w(t)$ and $v_N(t)=e^{-itH}u_0^{(N)}+w_N(t)=v_L^{(N)}(t)+w_N(t)$. In particular, for $\tau$ chosen as in Proposition~\ref{3propLocal}, of the form ~\eqref{eq:tauDef} we have $\|w_N\|_{Y^{\sigma}_{t_0,\tau}}, \|w\|_{Y^{\sigma}_{t_0,\tau}} \leqslant \lambda$. 

(\textit{ii}) By the triangle inequality and Lemma~\ref{lem.aa} we have
\[
    \|v(t)-v_N(t)\|_{X^{\sigma}} \leqslant \|u_0-u_0^{(N)}\|_{X^{\sigma}} + \|w(t)-w_N(t)\|_{Y^{\sigma}_{t_0,\tau}}\,.
\]
Which is bounded (in $N$). Thus it is sufficient to prove the convergence for $\sigma ' =0$ only, using interpolation to obtain the result in its full generality. To this end, we observe that $w-w_N$ satisfies
\[
    (i\partial_t-H)(w-w_N)=\cos(2t)^{-\alpha(p,d)}\left(F(v_L+w)-\mathbf{S}_NF(\mathbf{S}_Nv_{L}^{(N)} + \mathbf{S}_Nw_N)\right)\,,
\]
Let us set $z_N=\mathbf{S}_Nv_{L}^{(N)} + \mathbf{S}_Nw_N - (v_L+w)$, which satisfies
\begin{multline*}
    \|z_N(t)\|_{X^{0}} \leqslant \|(\operatorname{id}-\mathbf{S}_N)v_L(t)\|_{X^{0}}+\|\mathbf{S}_N(u_0-u_0^{(N)})\|_{X^{0}} \\ 
    + \|(\operatorname{id}-\mathbf{S}_N)w(t)\|_{X^{0}} + \|\mathbf{S}_N(w(t)-w_N(t))\|_{X^{0}}\,,
\end{multline*}
hence by the uniform continuity of $\mathbf{S}_N$ acting on $X^{0}$, and the assumption of the lemma we infer that
\begin{equation}
\label{eq:zN}
    \|z_N(t)\|_{X^{0}} \leqslant C\|w-w_N\|_{Y^{0}_{t_0,\tau}} + \varepsilon (N)\,,
\end{equation}
for some quantity $\varepsilon (N) \to 0$. We now observe that 
\begin{align*}
    (i\partial_t-H)(w-w_N)&=\cos(2t)^{-\alpha (p,d)}(\operatorname{id}-\mathbf{S}_N)F(v_L+w) \\
    &+ \cos(2t)^{-\alpha(p,d)} \mathbf{S_N}\left(F(v_L+w)-F(v_L+w+z_N)\right)\,,
\end{align*}
hence by the Bernstein estimates and by the uniform continuity of the $\mathbf{S_N}$, we obtain: 
\begin{align*}
    \|w-w_N\|_{Y^{0}_{t_0,\tau}} &\lesssim \cos(2t)^{-\alpha (p,d)} N^{-\sigma}\|F(v_L+w)\|_{\tilde{Y}^{\sigma}_{t_0,\tau}}\\
    & + \cos(2t)^{-\alpha (p,d)}\|\left(F(v_L+w)-F(v_L+w+z_N)\right)\|_{\tilde{Y}^0_{t_0,\tau}}\,.
\end{align*}
Therefore, using estimates~\eqref{3eqStab} and \eqref{3eqContract}, and up to choosing $\tau$ smaller (reducing the constant $C$ in $\tau$), but still of the form~\eqref{eq:tauDef}, we obtain
\[ 
    \|w-w_N\|_{Y^{0}_{t_0,\tau}} \leqslant CN^{-\sigma}\lambda + \frac{1}{2C}\|z_N\|_{X^{0}}\,, 
\]
where the constant $C$ is as in~\eqref{eq:zN}, and therefore we obtain that $\|w(t)-w_N(t)\|_{X^{0}} \to 0$, hence the result for $\sigma' =0$.

(\textit{iii}) We readily know that $\|w-w_N\|_{X^{\sigma'}}$ is bounded, thus by interpolation it remains to prove that 
\begin{equation}
    \|w-w_N\|_{Y^{0}_{t_0,\tau}} \lesssim N^{-\sigma}\lambda\,.
\end{equation}
The computations of (\textit{ii}) imply that for a well-chosen $\tau$ of the form~\eqref{eq:tauDef} we have
\[
    \|w-w_N\|_{Y^{0}_{t_0,\tau}} \leqslant CN^{-\sigma}\lambda + \frac{1}{2C}\|z_N\|_{X^{\sigma}}\,,
\]
but thanks to the assumption we have 
$\|z_N\|_{Y^0_{t_0,\tau}} \leqslant C\|w-w_N\|_{Y^{0}_{t_0,\tau}} + 2N^{-\sigma}\lambda$ which gives the claimed estimate. 
\end{proof}

Iterating (\textit{iii}) gives the following long-time estimate, provided a solution $v$ to~\eqref{3NLSH} has been constructed on an interval $[0,T]$. 

\begin{corollary}[Long-time perturbation] \label{3approxLemmaLong} 
Let $d\geqslant 2$, $p\in \left(1,1+\frac{4}{d}\right)$, $0<\sigma<\sigma ' <\sigma (p,d)$ and $\lambda, R >0$. Let $u_0$ (resp. $u_0^{(N)}$) in $B_{X^{\sigma'}}(R) \cap B_{X^{\sigma}}(\lambda)$ and let $v=v_L+w$ (resp. $v_N=v_L+w_N$) be the associated solution to~\eqref{3NLSH} (resp.~\eqref{HNLSN}) and assume that $v$ is well-defined on $[0,T]$. Then we have: 
\[
    \|w-w_N\|_{Y^{\sigma}_{[0,T]}} \leqslant CN^{\sigma - \sigma '}\,,
\]
where the constant $C$ may depend on $T, \lambda, R$ but not on $u_0$. 
\end{corollary}

\section{Quasi-invariant measures and their evolution}\label{3secEvol}

Our next task is to globalise the local statement of Proposition~\ref{3propLocal}. In order to do so we need to keep track of the measures of the good sets of well-posedness, \textit{i.e} we need to estimate $\mu(\phi_t(A))$ where $A$ is a $\mu$-measurable set and $\phi_t$ stands for the flow of \eqref{3NLSH}. 

\subsection{Quasi-invariant measure evolution bounds}

Due to explicit time dependence in the equation \eqref{3NLSH} we do not expect the measure formally defined by
\begin{equation}
    \nu_t(A)=\int_A e^{-\frac{1}{2}\|u\|^2_{\mathcal{H}^1} - \frac{\cos(2t)^{-\alpha(p,d)}}{p+1}\|u\|^{p+1}_{L^{p+1}}} \,\mathrm{d}u = \int_A e^{- \frac{\cos(2t)^{-\alpha(p,d)}}{p+1}\|u\|^{p+1}_{L^{p+1}}} \,\mathrm{d}\mu\,,
\end{equation}
to be invariant, but only quasi-invariant. This is one of the main ideas in~\cite{BT19} and we will carry out the same program. The quasi-invariance will be obtained using a Liouville theorem on finite dimensional approximations of~\eqref{3NLSH} and a limiting argument. 

The flow of the equation
\begin{equation}
\tag{HNLS${}_N$}
    \left\{
    \begin{array}{c}
         i\partial_t v-Hv=\cos(2t)^{-\alpha(p,d)} \mathbf{S}_N(|\mathbf{S}_Nv|^{p-1}\mathbf{S}_Nv)  \\
         v(0)=u_0\,, 
    \end{array}
    \right.
\end{equation}
will be denoted by $\phi_t^{N}$ and can be decomposed as follows: writing $v=v_N+v^{N}$, a solution to~\eqref{HNLSN} with $v_N=\Pi_Nv$ (where we recall that $\Pi_N$ is the (non-smooth) projection on modes $0\leqslant n\leqslant N$) we observe that from equation~\eqref{HNLSN}, $v_N$ satisfies
\begin{equation}
    \label{HNLSNN}
    \left\{
    \begin{array}{c}
         i\partial_t v_N-Hv_N=\cos(2t)^{-\alpha(p,d)} \mathbf{S}_N(|\mathbf{S}_Nv_N|^{p-1}\mathbf{S}_Nv_N)  \\
         v(0)=\mathbf{S}_Nu_0\,, 
    \end{array}
    \right.
\end{equation}
which is a finite dimensional ordinary differential equation in $\mathbb{C}^{N+1} \simeq \mathbb{R}^{2(N+1)}$ and denote by $\tilde{\phi}_t^N$ its flow. 

From  equation~\eqref{HNLSN}, we observe that $v^{N}$ satisfies 
\begin{equation}
\label{3eqHigh}
    \left\{
    \begin{array}{c}
         i\partial_t v^N-Hv^N=0  \\
         v(0)=(\operatorname{id}-\Pi_N)u_0\,, 
    \end{array}
    \right.
\end{equation}
so that if we identify $v^N$ to the sequence $(v_n)_{n >N}$ such that $v^N=\displaystyle \sum_{n>N}v_ne_n$, we can explicitly solve~\eqref{3eqHigh} and find that for every $n \geqslant N+1$,
\begin{equation}
    \label{3eqSolHigh}
    v_n(t)=e^{-it\lambda_n^2}(u_{0})_n\,,
\end{equation}
and denote by $\phi_t^{\perp,N}$ the flow of~\eqref{3eqHigh}. 

We start with a lemma which is nothing but the Liouville theorem, whose proof is recalled. 

\begin{lemma}\label{3lemmaInvariance} The measure $\mathrm{d}\mu_N$ is invariant under the flow $\tilde{\phi}^{N}_t$ of \eqref{HNLSNN}.
\end{lemma}

\begin{proof} First we recall that \eqref{HNLSNN} is locally well-posed in $\mathcal{C}^1(\mathbb{R},\mathbb{R}^{2(N+1)})$, thanks to the Cauchy theory for ordinary differential equations and globally well-posed since $t \mapsto \|u_N(t)\|_{L^2}$ is conserved. Moreover this equation admits a Hamiltonian structure. In order to see it, we write that for every $n$, $u_n=p_n+iq_n$ where $p_n, q_n$ are real numbers. Then we claim that there exists a function $E_N = E_N(t,p_1, \dots, p_N,q_1, \dots, q_N)$ such that if $p=(p_1, \dots, p_n)$ and $q:=(q_1, \dots, q_n)$, \eqref{HNLSNN} takes the form 
\begin{equation}
\label{3eqHamilton}
    \left\{
    \begin{array}{cc}
        p'(t) =& \partial_qE_N(t,p(t),q(t)) \\
        q'(t) =& -\partial _pE_N(t,p(t),q(t))\,.
    \end{array}
    \right.
\end{equation}
In order to find $E_N$, write $u=p+iq$ and equate real and imaginary parts of~\eqref{HNLSNN} to obtain that a Hamiltonian satisfying~\eqref{3eqHamilton} is given by
\[
    E_N(t,p_1, \dots, p_N,q_1,\dots, q_N) := \frac{1}{2}\sum_{n=0} \lambda_n^2(a_n^2+b_n^2) + \frac{\cos(2t)^{\frac{d}{2}(p-1)-2}}{p+1}\left\|\sum_{n=0}^N(a_n+ib_n) e_n\right\|^{p+1}_{L^{p+1}(\mathbb{R}^d)}\,.
\] 
More details are given in~\cite{burqThomannTzvetkov}, Lemma~8.1. For a solution $(p,q)$ to \eqref{3eqHamilton} we write $(p(t),q(t))=\tilde{\phi}_t^{N}(p_0,q_0)$. Let 
\[
    \mathrm{d}u_N\coloneqq\bigotimes_{n=0}^N \mathrm{d}u_n=\bigotimes_{n=0}^N (\mathrm{d}p_n\otimes \mathrm{d}q_n)\,,
\] 
then for every smooth function $f$ with compact support we have 
\[
    \frac{\mathrm{d}}{\mathrm{d}t} \int  f(\tilde{\phi}_t^N(p_0,q_0))\,\mathrm{d}u_N (p_0,q_0) = \int (\partial_pE_N \partial _qf - \partial_qE_N\,\partial _pf)\,\mathrm{d}p\,\mathrm{d}q=0\,,
\] 
where we integrated by parts in the last equality. Finally by density this shows that the measure~$\mathrm{d}u_N$ is invariant under $\tilde{\phi}_t^{N}$, so is $\mathrm{d}\mu_N$. 
\end{proof}

The Hamiltonian we have found takes the form: 
\[
    E_N(t)\coloneqq\frac{1}{2}\|v(t)\|_{\mathcal{H}^1}^2+ \frac{\cos(2t)^{-\alpha(p,d)}}{p+1}\|\mathbf{S}_Nv(t)\|_{L^{p+1}}^{p+1}\,.
\] 
It is \textit{not conserved} under the flow $\phi_t^N$ and more precisely we have
\begin{equation}
    \label{3diffenergy}
    E_N'(t) = \frac{d(p-1)-4}{p+1}\tan (2t)\cos(2t)^{-\alpha(p,d)}\|\mathbf{S}_Nv(t)\|_{L^{p+1}}^{p+1}\,.
\end{equation}

For $t\geqslant 0$ and $N\geqslant 0$ we define the finite measures, which are not necessarily probability measures, associated to the non-conserved energies $E_N$ by:
\[
    \nu_{t}^{(N)}(A)\coloneqq \int_Ae^{- \frac{\cos(2t)^{-\alpha(p,d)}}{p+1}\|\mathbf{S}_Nu\|_{L^{p+1}}^{p+1}}\,\mathrm{d}\mu(u)\,.
\]
From the definition it follows that for all $\mu$-measurable sets $A$ one has $\nu_{t}^{(N)}(A) \leqslant \mu(A)$. Moreover we have the following convergence result. 

\begin{lemma} Let $t \geqslant 0$. Then the measure $\nu_t$ is not trivial, \textit{i.e.} its density with respect to $\mu$ does not vanish almost surely. Moreover we have the strong convergence $\nu_{t}^{(N)} \to \nu_t$, that is for every measurable set $A$, $\nu_t^{(N)}(A) \to \nu_t(A)$. 
\end{lemma}

\begin{proof} To see the first claim, just observe that for $u$ in the support of the measure $\mu$, since $p+1 \in \left(2,2+\frac{4}{d}\right)$ we have $p+1<\frac{2d}{d-2}$ and then $\|u\|_{L^{p+1}}<\infty$ thanks to Lemma~\ref{3lemmaGainLp}. Moreover for such $u$ in the support of $\mu$ we have $\mathbf{S}_Nu \to u$ in $L^{p+1}$  as $N \to \infty$. The domination $e^{-\frac{\cos(2t)^{-\alpha(p,d)}}{p+1}\|\mathbf{S}_Nu\|_{L^{p+1}}^{p+1}} \leqslant 1$ and Lebesgue's convergence theorem ensures the strong convergence $\nu_t^{(N)} \to \nu_t$.
\end{proof}

The quantitative quasi-invariance property which we are going to state and prove are exactly the same as in~\cite{BT19}. However for convenience we recall the proof.

\begin{proposition}[Measure Evolution]\label{3propMeasureEvol} Let $t\in (-\frac{\pi}{4},\frac{\pi}{4})$, $N\geqslant 1$ and a $\mu$-measurable set $A$.
\begin{enumerate}[label=(\textit{\roman*})]
    \item $(\phi^N_t)_*\mu$ and $\mu$ are mutually absolutely continuous with respect to each other;
    \item $\nu_t \left(\phi_t^N(A)\right) \leqslant \nu_0(A)^{\cos(2t)^{\alpha(p,d)}}$;
    \item $\nu_0(A) \leqslant \nu_t\left(\phi_t^N(A)\right)^{\cos(2t)^{\alpha(p,d)}}$.
\end{enumerate}
\end{proposition}

\begin{proof} Note that once we have proved (\textit{ii}) and (\textit{iii}) then we immediately can conclude the proof for (\textit{i}) since the two previous points give us that \[(\phi_t^{N})_*\nu_t \ll \nu_{0} \ll (\phi_t^N)_* \nu_t\,.\] As by definition $\mu \ll \nu_t$ and $\nu_t \ll \mu$, which proves (\textit{i}).

To prove (\textit{ii}) we start by studying the measure $\nu_{t}^{N}$, writing 
\begin{align*}
    \frac{\mathrm{d}}{\mathrm{d}t} \nu_t^{(N)}(\phi^N_tA) &= \frac{\mathrm{d}}{\mathrm{d}t} \left(\int_{v \in \phi^N_tA} e^{-\frac{\cos(2t)^{-\alpha(p,d)}}{p+1} \|\mathbf{S}_Nv\|_{L^{p+1}}^{p+1}} \,\mathrm{d}\mu(v)\right)\\
    &= \frac{\mathrm{d}}{\mathrm{d}t} \left(\int_A e^{-\frac{\cos(2t)^{\alpha(p,d)}}{p+1} \|\mathbf{S}_N\phi_t^{(N)}u_0\|_{L^{p+1}}^{p+1}}\,\mathrm{d}\mu(u_0)\right) \\
    &=\frac{d(p-1)-4}{2}\tan(2t)\int_A\alpha_N (t,u) e^{-\alpha_N(t,u)}\,\mathrm{d}\mu(u_0)\,,
\end{align*}
with $\alpha_N (t,u)=\frac{\cos(2t)^{-\alpha(p,d)}}{p+1}\|\mathbf{S}_Nu(t)\|_{L^{p+1}}^{p+1}$. In the first equality we have used the change of variable $v=\phi^N_tu_0$, which leaves $\mathrm{d}\mu$ invariant according to Lemma~\ref{3lemmaInvariance}, indeed since $\mu = \mu_N \otimes \mu^{\perp}_N$ and $\phi_t^N = \tilde{\phi}_t^N \otimes e^{-itH}$, the invariance follows from the invariance of $\mu_N$ under $\tilde{\phi}_t^N$ (Application of Lemma~\ref{3lemmaInvariance} and conservation of $E_N(t)$) and invariance of $\mu_N^{\perp}$ under $e^{-itH}$, which is just invariance of complex Gaussian random variables under rotation. 

Next we apply Hölder's inequality with a parameter $k\geqslant 1$ to be chosen later, and use that for all positive $\alpha$ we have $\alpha^k e^{-\alpha} \leqslant k^ke^{-k}$ which gives
\begin{align*}
    \frac{\mathrm{d}}{\mathrm{d}t} \nu_t^{(N)}(\phi^N_tA)
    &=|d(p-1)-4| \tan(2t) \left(\int_A\alpha_N^k(t,u(t))e^{-\alpha_N (t,u)} \,\mathrm{d}\mu(u_0)\right)^{\frac{1}{k}} \left(\nu^{(N)}_t(\phi_t^NA)\right)^{1-\frac{1}{k}}\\
    & \leqslant |d(p-1)-4| \tan(2t) \frac{k}{e} \left(\nu^{(N)}_t(\phi_t^NA)\right)^{1-\frac{1}{k}}\,,
\end{align*}
where we used the backward change of variable that leaves the measure invariant again. Now we choose $k$ to optimise this inequality, namely $k:=-\log \left(\nu_t^{N}(\phi_t^NA)\right)$ so that
\[
    \frac{\mathrm{d}}{\mathrm{d}t} \nu^{(N)}_t(\phi_t^NA) \leqslant |d(p-1)-4| \tan(2t) \log\left(\nu^{(N)}_t(\phi_t^NA)\right)\nu^{(N)}_t(\phi_t^NA)\,.
\]
We rewrite it as:
\begin{align*}
    \label{3backward}
    -\frac{\mathrm{d}}{\mathrm{d}t}\left(\log \left(-\log \left(\nu^{(N)}_t(\phi_t^NA)\right)\right)\right) & \leqslant |d(p-1)-4|\tan (2t) \\
    &= - \left\vert \frac{d}{2}(p-1)-2\right\vert\frac{\mathrm{d}}{\mathrm{d}t} \left(\log (\cos (2t))\right)\,,\notag
\end{align*}
which after integration reads
\[
    -\log\left(\nu^{(N)}_t(\phi_t^NA)\right) \leqslant (\nu^{(N)}_0(A))^{\cos (2t)^{\alpha(p,d)}}\,.
\]
Then we observe that for $M \geqslant N$ and for every $\mu$-measurable set $A$, one has \[\nu^{(N)}_t(\phi_t^NA)=\nu_t^{(M)}(\phi_t^NA) \to \nu_t(\phi_t^NA) \text{ as } M \to \infty\] so that finally get the result passing to the limit.

The estimate (\textit{iii}) is obtained by similar means, observing first that \[\frac{\mathrm{d}}{\mathrm{d}t} \nu_t^{(N)}(\phi^N_tA) \geqslant - |d(p-1)-4| \tan(2t) \frac{k}{e} \left(\nu^{(N)}_t(\phi_t^NA)\right)^{1-\frac{1}{k}}\,,\] for every $k \geqslant 1$, optimising in $k$ and integrating as before. 
\end{proof}

\subsection{Construction of the flow on a full probability set}

The purpose of this section is simply to construct a set $\Sigma$ such that $\mu_0(\Sigma)=1$ and such that for any $u_0 \in \Sigma$ the flow of~\eqref{3NLSH} is well-defined.

\begin{proposition}[Long-time probabilistic Cauchy theory]\label{prop:globalWP} Let $d\geqslant 2$, $p\in\left(1,1+\frac{4}{d}\right)$ and $\sigma \in (0,\sigma(p,d))$. For $i\geqslant i_0$ sufficiently large, the following holds. There exists a set $\Sigma ^i \subset X^{\sigma}$, such that: 
\begin{enumerate}[label=(\textit{\roman*})]
    \item $\Sigma^i$ is closed in the $X^{\sigma}$ topology. 
    \item $\nu_0(X\setminus \Sigma^i) \leqslant e^{-ci}$ for some absolute constant $c>0$. 
    \item For any $u_0 \in \Sigma ^i$, there exists a solution $v=e^{-itH}u_0+w = v_L+w \in v_L + \mathcal{C}^0((-\frac{\pi}{4},\frac{\pi}{4}),\mathcal{H}^{\sigma})$ which satisfies, for all $t \in (-\frac{\pi}{4},\frac{\pi}{4})$, the following: 
    \begin{equation}
        \label{eq:boundXsigma}
        \|v(t)\|_{X^{\sigma}} \leqslant C\sqrt{i} \left(\frac{\pi}{4}-|t|\right)^{-\frac{\alpha(p,d)}{2}} \left\vert\log \left(\frac{\pi}{4}-|t|\right)\right\vert^{\frac{1}{2}}\,,
    \end{equation}
    and also
    \begin{equation}
        \label{eq:boundHsigma}
        \|w(t)\|_{\mathcal{H}^{\sigma}}\leqslant C\sqrt{i} \left(\frac{\pi}{4}-|t|\right)^{-\frac{\alpha(p,d)}{2}} \left\vert\log \left(\frac{\pi}{4}-|t|\right)\right\vert^{\frac{1}{2}}\,,
    \end{equation}
    where $C>0$ does not depend on $i, \sigma$. 
\end{enumerate}
\end{proposition}

\begin{remark} We will sometimes need a useful consequence of the proof of Proposition~\ref{prop:globalWP}, that is, for an initial data $u_0 \in \Sigma ^i$, $t_0\in (0,\frac{\pi}{4})$, and $\varepsilon_0>0$ (arbitrarily small), if $\tau\sim \left(\frac{\pi}{4}-t_0\right)^{K+L(\frac{\alpha(p,d)}{2}+\varepsilon_0)}$ (the constants $K, L$ being that of~\eqref{eq:tauDef}), then there holds
\begin{equation}
    \label{eq.growthY}
    \|w\|_{Y^{\sigma}_{t_0,\tau}} \lesssim \sqrt{i} \left(\frac{\pi}{4}-t_0\right)^{-\frac{\alpha(p,d)}{2}-\varepsilon_0}\,.
\end{equation}
\end{remark}

We immediately infer the following consequence.

\begin{corollary}\label{3coro1}
Let $d\geqslant 2$, $p\in\left(1,1+\frac{4}{d}\right)$ and $\sigma \in (0,\sigma(p,d))$. There exists a set $\Sigma$ of full $\nu_0$ measure and such that for any $u_0\in \Sigma$, a global solution $v=e^{-itH}u_0+w=v_L+w$ (of the form given by Proposition~\ref{3propLocal}) exists on $(-\frac{\pi}{4},\frac{\pi}{4})$ and satisfies
\begin{equation}
    \label{eq:boundGrowthPolynomial}
    \|v(t)\|_{X^{\sigma}}, \|w(t)\|_{\mathcal{H}^{\sigma}} \leqslant C(u_0) \left(\frac{\pi}{4}-|t|\right)^{-\frac{\alpha(p,d)}{2}} \left\vert\log \left(\frac{\pi}{4}-|t|\right)\right\vert^{\frac{1}{2}}\,,
\end{equation}
where $\mu(C(u_0)>\lambda)\lesssim e^{-c\lambda^2}$. 
\end{corollary}

\begin{proof} This set is just $\Sigma \coloneqq \displaystyle\bigcup_{i\geqslant i_0}\Sigma^i$ which satisfies the required properties and has measure 
\[
    \nu_0(X\setminus \Sigma) \leqslant \inf_{i \geqslant 1} \nu_0(X\setminus \Sigma^i)=0\,.\qedhere 
\]
\end{proof}

In order to prove this result we will need the following corresponding statement for~\eqref{HNLSNN}. 

\begin{lemma}[Uniform in $N$ long-time probabilistic Cauchy theory]\label{lem.uniformGLobalN} Let $d\geqslant 2$, $p\in\left(1,1+\frac{4}{d}\right)$ and $\sigma \in (0,\sigma(p,d))$. Let $i, j \geqslant 1$, and let $T_j=\frac{\pi}{4}(1-e^{-j})$. Let $N\geqslant 1$. Then there exists a set $\Sigma^{i,j}_N$ such that 
\begin{enumerate}
    \item $\nu_0(X \setminus \Sigma^{i,j}_N) \leqslant Ce^{-cij}$, at least for $i\geqslant i_0, j\geqslant j_0$ large enough. The constants $C$ and $c$ do not depend on $i, j$. 
    \item For any $u_0^{(N)} \in \Sigma^{i,j}_N$, the associated solution $v_{N}=v_{L}^{(N)}+w_N$ to~\eqref{HNLSN} satisfies, for all $t\in [-T_j,T_j]$,
    \[
        \|v_N(t)\|_{X^{\sigma}}, \|w_N(t)\|_{\mathcal{H}^{\sigma}} \leqslant C \sqrt{ij} e^{\frac{j\alpha(p,d)}{2}}\,,
    \]
    where the constant $C>0$ does not depend on $i, j, N$.
\end{enumerate}
\end{lemma}

\begin{proof} 
By symmetry we deal only with positive times $[0,T_j]$. Let $\lambda \coloneqq C_* \sqrt{ij} e^{\frac{j\alpha(p,d)}{2}}$, where $C_*$ needs to be adjusted in the proof. Thanks to Proposition~\ref{3propLocal}, there is a constant $C_1$ such that if we set $\tau = C_1e^{-jK}\lambda^{-L}$ (the constants $K, L$ respecting being as in~\eqref{eq:tauDef}), then the intervals $[n\tau, (n+1)\tau]$ for $n \in\left\{0, \dots , \frac{T_j}{\tau}\right\}$ are local well-posedness intervals for~\eqref{HNLSN} and initial data in $B_{X^{\sigma}}(\lambda)$, uniformly in $N$, thanks to Remark~\ref{rem:localTheory}. Therefore, defining the set
\[
    \Sigma^{i,j}_N \coloneqq \bigcap_{n=0}^{\left\lfloor T_j/\tau\right\rfloor} (\phi_{n\tau}^N)^{-1}(B_{X^{\sigma}}(\lambda))\,,
\]
it follows that any initial data $u_0^{(N)}$ gives rise to an $X^{\sigma}$ solution on $[-T_j,T_j]$ which matches the bounds~\eqref{eq:boundXsigma} and~\eqref{eq:boundHsigma}. It remains to estimate $\nu_0(X\setminus\Sigma_N^{i,j})$, which by application of~\eqref{3propMeasureEvol} gives: 
\begin{equation*}
    \nu_0(X\setminus \Sigma_N^{i,j}) \leqslant \sum_{n=0}^{\left\lfloor T_j/\tau\right\rfloor} \nu_0((\phi_{n\tau}^N)^{-1}(X\setminus B_{X^{\sigma}}(\lambda))) \leqslant \sum_{n=0}^{\left\lfloor T_j/\tau\right\rfloor} \nu_{n\tau}(X\setminus B_{X^{\sigma}}(\lambda))^{\cos (2n\tau)^{\alpha (p,d)}}\,. 
\end{equation*}
Then we use the crude bound $\nu_{n\tau}(X\setminus B_{X^{\sigma}}(\lambda)) \leqslant \mu_0(X\setminus B_{X^{\sigma}}(\lambda))\leqslant e^{-c\lambda ^2}$, so that finally 
\[
    \nu_0(X\setminus \Sigma_N^{i,j})\leqslant \frac{T_j}{\tau} e^{-c\lambda^2 \cos(2T_j)^{\alpha(p,d)}} \lesssim e^{jK}\lambda^Le^{-c\lambda^2e^{-j\alpha(p,d)}}\,,
\]
so that by definition of $\lambda$ this yields 
\[
    \nu_0(X\setminus \Sigma_N^{i,j}) \lesssim e^{j(K+L\alpha(p,d)/2)}e^{-cC_*^2ij}\,,
\]
which is less that $Ce^{-cij}$ for large enough $C_1$ and for $i \geqslant i_0$, $j \geqslant j_0$. 
\end{proof}

\begin{proof}[Proof of Proposition~\ref{prop:globalWP}] Let us set $\Sigma^i_N = \displaystyle\bigcap_{j\geqslant j_0} \Sigma_N^{i,j}$ which is such that
\[
    \nu_0(X\setminus \Sigma_N^i) \lesssim e^{-ci}\,,
\]
and such that any $u_{0,N} \in \Sigma^i_N$ gives rise to a global solution satisfying the bounds~\eqref{eq:boundHsigma} and~\eqref{eq:boundXsigma}. We consider the set 
\[
    \Sigma^i = \{u_0 \in X^{\sigma}, \exists N_k \to \infty \text{ and } u_{0}^{N_k} \in \Sigma_{N_k}^i, \|u_0-u_0^{(N_k)}\|_{X^{\sigma}}\to 0\}\,,
\]
which is a closed subset of $X^{\sigma}$, proving (\textit{i}). 

To prove (\textit{ii}), observe that $\limsup_N \Sigma_N^{i} \subset \Sigma_i$ thus by Fatou's lemma there holds
\[
    \nu_0(\Sigma^i) \geqslant \nu_0(\limsup_N \Sigma_N^{i}) \geqslant \limsup_N \nu_0(\Sigma_N^{i}) \geqslant \nu_0(X)-e^{-ci}\,.
\]

To prove (\textit{iii}), let $u_0 \in \Sigma^i$ and an associated sequence $u_0^{(N_k)}\in \Sigma_{N_k}^i$ such that 
\[
    \|u_0-u_0^{(N_k)}\|_{X^{\sigma}} \to 0\,.
\]
We need to construct a solution $v=v_L+w=e^{-itH}u_0+w$ to~\eqref{3NLSH} satisfying the bounds~\eqref{eq:boundXsigma} and~\eqref{eq:boundHsigma}. To this end, let $T$ such that $|T|<\frac{\pi}{4}$, $j\geqslant 1$ such that $T_j \sim T$, and let $\tau$ be a local well-posedness time associated to $u_0$ and $T_j$. Without restriction, only considering large $k$, we can even assume that this $\tau$ is also a local well-posedness time for the $u_{0}^{(N_k)}$. By Lemma~\ref{3approxLemma}~(\textit{ii}) we infer that $\|v(\tau)-v_{N_k}(\tau)\|_{X^{\sigma '}} \to 0$ for any $0\leqslant\sigma '< \sigma$. We also know that by definition of $u_0^{(N_k)}$ and Lemma~\ref{lem.uniformGLobalN} we have a uniform bound for $\|v_N(\tau)\|_{X^{\sigma}}$, and since $v_N(\tau)$ converges to $v(\tau)$ in $X^{\sigma'}$, it follows that $v(\tau) \in X^{\sigma}$ with the same bound as the one enjoyed by the $v_N(\tau)$. This argument can be iterated for $n \leqslant \lfloor T/\tau \rfloor$. Since $T$ is arbitrary, this yields (\textit{iii}).  
\end{proof}

With this first global theory, we are able to pass to the limit $N\to \infty$ in Proposition~\ref{3propMeasureEvol}. 

\begin{proposition}\label{prop:quasiInvariantBound} Let $A \subset \Sigma$ be a measurable set, then there holds that for all $t \geqslant 0$, 
\begin{equation}
    \label{eq:boundQuasiInvariant}
    \nu_0(A) \leqslant \nu_t(\phi_tA)^{\cos(2t)^{\alpha(p,d)}}\,.
\end{equation}
\end{proposition}

\begin{proof} Let $0<\sigma <\sigma (p,d)$. We recall that since $\nu_0(X\setminus X^{\sigma})=0$, and since $X^{\sigma}$ is a complete space, $\nu_0$ can be viewed as a regular measure on $X^{\sigma}$ (this is Ulam's theorem, see for example Theorem~7.1.4 in \cite{dudley}). This remark will allow us to reduce the analysis to compact sets. In our analysis we will need a $\sigma ' \in (\sigma, \sigma (p,d))$, which we fix. We also further restrict to proving~\eqref{eq:boundQuasiInvariant} for all $t\in [0,T]$, $T$ being arbitrary in $(0,\frac{\pi}{4})$. 

Let us first consider a compact set $K \subset X^{\sigma}$. Using that $\Sigma = \displaystyle\bigcup_{i \geqslant 1} \Sigma ^i$, and $\nu_0(X^{\sigma '})=1$, we can write that 
\[
    \nu_0(K)= \lim_{R\to \infty}\lim_{i\to \infty} \nu_0(K\cap \Sigma^i \cap B_{X^{\sigma'}}(R))\,. 
\]
Since $B_{X^{\sigma'}}(R)$ is closed in $X^{\sigma}$ and $\Sigma ^i$ is closed in $X^{\sigma}$, the set $K_{R,i} \coloneqq K \cap \Sigma^i \cap B_R$ is a compact of $X^{\sigma}$. Let $\varepsilon >0$. By Corollary~\ref{3approxLemmaLong} we can fix $N_0(\varepsilon)>0$ such that for all $N\geqslant N_0(\varepsilon)$ there holds
\[
    \phi_t^N(K_{R,i})\subset \phi_t K_{R,i} + B_{X^\sigma}(\varepsilon)\,.
\]
Let us consider a local existence time $\tau$ for $v$ in $X^{\sigma}$, associated to the initial data in $K_{R,i}$. We recall that such a local well-posedness time depends only on $T$ and $\|u_0\|_{X^{\sigma}}$, which is bounded by $R$. But thanks to Proposition~\ref{3propLocal} (\textit{iii}) we also know that $v$ exists on $[0,\tau]$ as a $X^{\sigma '}$ solution, satisfying $\|v\|_{L^{\infty}_{[0,\tau]}X^{\sigma '}} \leqslant CR$. Therefore $\phi_{t}K_{R,i}$ is bounded in $X^{\sigma '}$ for all $t\in[0,\tau]$. We can iterate this argument on $[\tau,2\tau], \dots, [(k-1)\tau, k\tau]$ where $k=\lfloor T/\tau\rfloor$, to obtain that $\phi_tK_{R,i}$ is bounded in $X^{\sigma '}$ and therefore compact in $X^{\sigma}$. 

We are now ready to give the proof. We start by applying Proposition~\ref{3propMeasureEvol} to bound
\[
    \nu_0(K_{R,i}) \leqslant \nu_t(\phi_t^N(K_{R,i}))^{\cos(2t)^{\alpha (p,d)}}\,.
\]
Observe that for $N \geqslant N(\varepsilon)$ we have  
\[
    \nu_t(\phi_t^N(K_{R,i})) \leqslant \nu_t(\phi_t K_{R,i} + B_{X^\sigma}(\varepsilon))\,,
\]
so that passing to the limit $\varepsilon \to 0$ gives 
\[
    \nu_0(K_{R,i}) \leqslant \nu_t(\phi_t K_{R,i})^{\cos(2t)^{\alpha (p,d)}} \leqslant \nu_t(\phi_tK)^{\cos(2t)^{\alpha (p,d)}}\,,
\]
because $\phi_tK_{R,i}$ is closed in $X^{\sigma}$.  

Passing to the limit $R \to \infty$ and $i\to \infty$ gives~\eqref{eq:boundQuasiInvariant} for any compact set $K \subset \Sigma$. To recover the full result, let $A \subset \Sigma ^i$ be a measurable set. Then by regularity of the measure $\nu_0$ there exists a sequence of compact sets $K_n \subset A$ such that $\nu_0(K_n) \to \nu_0(A)$. Applying~\eqref{eq:boundQuasiInvariant} to $K_n$ yields that 
\[
    \nu_0(K_n)\leqslant \nu_t(\phi_t K_n)^{\cos(2t)^{\alpha(p,d)}} \leqslant \nu_t(\phi_tA)^{\cos(2t)^{\alpha(p,d)}}\,,
\]
and passing to the limit $n\to \infty$ yields the result. 
\end{proof}

\section{Scattering }\label{3sec:global}

\subsection{Growth of the \texorpdfstring{$L^{p+1}$}{Lebesque} norms}

In order to estimate the $L^{p+1}$ norms of a solution $v$ to~\eqref{3NLSH} we first state a result valid for a given time. 

\begin{lemma} Let us define, for any $\lambda >0$ and any $t\in \left(\frac{\pi}{4},\frac{\pi}{4}\right)$ the set 
\begin{equation}    
    \label{3eqAtlambda}
    A_{t,\lambda}:=\{u_0\in X^0, \; \|\phi_t u_0\|_{L^{p+1}}> \lambda\}\,. 
\end{equation}
There holds:  
\begin{equation}
    \label{3measureEvolCoro2}
    \nu_0(A_t) \leqslant e^{-\frac{\lambda^{p+1}}{p+1}}\,.
\end{equation}
\end{lemma}

\begin{remark} The estimate~\eqref{3measureEvolCoro2} is better than the bound given by $\nu_0(A)\leqslant \mu (A)^{\cos (2t)^{\alpha(p,d)}}$, which amounts to saying that $\nu_0(A) \leqslant 1$ when $p\sim 1+\frac{4}{d}$.  
\end{remark}

\begin{proof}
Let us prove~\eqref{3measureEvolCoro2}. Using~\eqref{eq:boundQuasiInvariant}, and remarking that by definition of $A_t$ we have $\|\phi_t u_0\|_{L^{p+1}}>\lambda$ for every $u_0 \in A_t$ so that:
\begin{align*}
    \nu_0(A) & \leqslant \nu_t(\phi_tA)^{\cos(2t)^{\alpha(p,d)}} \\
    & \leqslant \left(e^{-\cos(2t)^{-\alpha(p,d)}\frac{\lambda^{p+1}}{p+1}}\mu (\phi_tA)\right)^{\cos(2t)^{\alpha(p,d)}}\\
    &\leqslant e^{-\frac{\lambda^{p+1}}{p+1}}\,,
\end{align*}
where we used that $\mu (\phi_tA) \leqslant 1$ in the last line.
\end{proof}

Next we want to estimate the $L^{p+1}$ norm of the solutions for all times. We will distinguish two cases and introduce a real number $p_{\text{max}}$ defined by 
\begin{equation}
    \label{eq:defPmax1}
    p_{\text{max}}(d):=\frac{5-d+\sqrt{9d^2-2d+9}}{2(d-1)}<1+\frac{3}{d-2}
\end{equation}
if $d \leqslant 7$. If $d \geqslant 8$ then 
\begin{equation}
    \label{eq:defPmax2}
    p_{\text{max}}(d) \coloneqq \min\{x>0, P_d(x)=0\}\,,
\end{equation} where
\[
    P_d\coloneqq(d-2)X^3+(d-4)X^2-6X-2d-4\,,
\]
is a polynomial with only one real root. 

Note that for $d\geqslant 8$ one has 
\[
    p_{\text{max}}(d) \geqslant \min\left\{\frac{5-d+\sqrt{9d^2-2d+9}}{2(d-1)},1+\frac{3}{d-2}\right\}\,.
\]
To prove this fact, we observe that the discriminant of $P_d$ is negative, at least for $d \geqslant 8$ and thus $P_d$ has a unique real root. Note that $\min\{\frac{5-d+\sqrt{9d^2-2d+9}}{2(d-1)},1+\frac{3}{d-2}\} = 1+\frac{3}{d-2}$ as soon as $d\geqslant 9$. To conclude we need to show that $P_d(1+\frac{3}{d-2})<0$ which is equivalent to $d^2-10d+7>0$, satisfied for $d\geqslant 9$. Similarly one has $P_d(\frac{5-d+\sqrt{9d^2-2d+9}}{2(d-1)})<0$ for $d=8$. 
We also have $p_{\text{max}}(d) \geqslant 1+ \frac{4}{d}$ if $d \leqslant 7$ with similar computations. 

\begin{proposition}\label{3lemGainLpp1} Let $d\geqslant 2$, $p \in (1,p_{\operatorname{max}}(d))$ and $\sigma \in (0,\sigma(p,d))$. Let also $\eta >0$ and $|t| < \frac{\pi}{4}$. There exists a set $G_{\eta,t} \subset X^0$ such that $\nu_0(X^0 \setminus G_{\eta,t})\leqslant \eta$ and such that for all $u_0 \in G_{\eta,t}$, there exists a unique solution to~\eqref{3NLSH} with initial data $u_0$ which writes $v(t')=e^{-itH}u_0+w(t')$ where $w \in Y^{\sigma}_{[-t,t]}$. Furthermore this solution satisfies for all $t'\in [-t,t]$ the estimate
\begin{equation}
    \label{3eqLpwithoutSobolev}
    \|w(t')\|_{L^{p+1}} \lesssim \log^{\frac{1}{2}}\left(\frac{1}{\eta}\right)  \left\vert \log \left(\frac{\pi}{4}-|t|\right)\right \vert ^{\frac{1}{2}}\,,
\end{equation}
where the implicit constant only depend on $p$ and $d$. 
\end{proposition}

From there we can deduce exactly as in Corollary~\ref{3coro1} the following global estimate. 

\begin{corollary}\label{3coro2} Let $d\geqslant 2$ and $p\in (1,p_{\text{max}}(d))$. There exists $\varepsilon_0 >0$ and a set $G$ of full $\nu_0$ measure, such that for all $u_0 \in G$ and all $\varepsilon \in [0,\varepsilon_0]$ there exists a unique global solution $v(t)=e^{-itH}u_0+w(t)$ to~\eqref{3NLSH} with initial data $u_0$ which furthermore satisfies
\begin{equation}
    \label{3eqLpwithoutSobolev2}
    \|w(t)\|_{\mathcal{W}^{\varepsilon ,p+1}} \lesssim C(u_0)\left\vert \log \left(\frac{\pi}{4}-|t|\right)\right \vert ^{\frac{1}{2}}\,,
\end{equation}
for every $|t|<\frac{\pi}{4}$. Moreover there exist $c,C>0$ such that 
\[
    \mu (C(u_0)>\lambda) \lesssim e^{-c \lambda ^2}\,.
\] 
\end{corollary}

\begin{proof} We only rapidly explain how we obtain the $\mathcal{W}^{\varepsilon, p+1}$ estimates. This follows from the proof of Lemma~\ref{3lemGainLpp1}, where we can multiply equation~\eqref{3eqLpp1} by $H^{\frac{\varepsilon}{2}}$, then all the estimates are similar provided $\varepsilon>0$ is small enough. In particular we only use non-endpoints estimates, for which there is always room for an $\varepsilon$ modification of the parameters.

Let us set $G \coloneqq \limsup_{n \to \infty} G_{2^{-n},\frac{\pi}{4}(1-2^{-n})}$. Then any $u_0 \in G$ belongs to infinitely many $G_{2^{-n},\frac{\pi}{4}(1-2^{-n})}$ therefore producing global solutions with required estimats. Finally observe that 
\[
    \nu_0(G) = \nu_0\left(\limsup_{n\to\infty}G_{2^{-n},\frac{\pi}{4}(1-2^{-n})}\right) \geqslant \limsup_{n\to \infty}\nu_0(G_{2^{-n},\frac{\pi}{4}(1-2^{-n})}) =\nu_0(X)\qedhere
\]
\end{proof}

\begin{proof}[Proof of Proposition~\ref{3lemGainLpp1}] Again we only deal with the forward in time part of the estimate, by time reversibility. Let $0<t<\frac{\pi}{4}$, $\eta >0$ and $p$ as in the lemma. Let also $\tau>0$ be a local existence time of the form~\eqref{eq:tauDef}, given by Proposition~\ref{3propLocal}, which we choose uniformly in $[0,t]$, and which we may shrink in the sequel. We also claim that we can prove the result on $v$ rather than $w$ since estimates of this kind on the linear part $v_L$ are granted by Lemma~\ref{3coroProba}. We may also assume, without restriction that we only consider initial data in $u_0 \in \Sigma$ so that we have access to estimate~\eqref{eq:boundGrowthPolynomial}, and readily solve the equation on $[-t,t]$.

We start with the case $p>1+\frac{2}{d-1}$. We set $t_n:=n\tau$ for $n=0, \dots, \lfloor \frac{t}{\tau}\rfloor$. Let $t' \in [t_n,t_{n+1})$ and write $v(t')$ in terms of $v(t_n)$ as the solution of the initial value problem at $t_n$ using the Duhamel formula:
\begin{equation}
    \label{3eqLpp1}
    v(t')=e^{-i(t'-t_n)H}v(t_n) - i \int_{t_n}^{t'} e^{-i(t'-s)H} \left(\cos (2s)^{-\alpha(p,d)}|v(s)|^{p-1}v(s)\right) \,\mathrm{d}s\,.
\end{equation}
The triangle inequality and the dispersion estimates $L^{p+1} \to L^{\frac{p+1}{p}}$ given by Proposition~\ref{3propStrichartz} yield
\begin{align}
    \label{3eqvtt0}
    \|v(t')\|_{L^{p+1}} &\leqslant \|v(t_n)\|_{L^{p+1}} +\underbrace{\|(e^{-i(t'-t_n)}-\operatorname{id})v(t_n)\|_{L^{p+1}}}_{I} \\
    &+ \underbrace{\int_{t_n}^{t'} \frac{1}{|t'-s|^{d \left(\frac{1}{2}- \frac{1}{p+1}\right)}} \cos (2s)^{-\alpha(p,d)} \|v(s)\|^{p}_{L^{p+1}}\,\mathrm{d}s}_{II}\notag\,.
\end{align}
In the sequel we estimate the terms $I$ and $II$ differently. 

For $II$ we use that 
\[
    \cos (2s) \gtrsim \left(\frac{\pi}{4}-s\right)
\] for $s \in [t_n,t_{n+1}]$. Then for parameters $\gamma, \gamma ' \geqslant 1$ satisfying $\frac{1}{\gamma}+\frac{1}{\gamma'}=1$ and Hölder's inequality:
\begin{align*}
    II &\lesssim \left(\frac{\pi}{4}-t\right)^{-\alpha(p,d)} \int_{t_n}^{t'} \frac{1}{|t'-s|^{d \left(\frac{1}{2}-\frac{1}{p+1}\right)}} \|v(s)\|^p_{L^{p+1}}\,\mathrm{d}s \\
    & \lesssim \left(\frac{\pi}{4}-t\right)^{-\alpha(p,d)} \left( \int_{t_n}^{t'} |t'-s|^{-d\gamma \left(\frac{1}{2}-\frac{1}{p+1}\right)}\,\mathrm{d}s\right)^{\frac{1}{\gamma}} \left( \int_{t_n}^{t'} \|v(s)\|_{L^{p+1}}^{p\gamma '}\,\mathrm{d}s\right)^{\frac{1}{\gamma '}} \\
    & \lesssim \left(\frac{\pi}{4}-t\right)^{-\alpha(p,d)} \tau^{\frac{1}{\gamma}-d \left(\frac{1}{2}-\frac{1}{p+1}\right)} \|v\|^p_{L^{p\gamma '}_{[t_n,t_{n+1}]}L^{p+1}}\,,
\end{align*}
provided the integrability condition $\gamma d \left(\frac{1}{2}-\frac{1}{p+1}\right)<1$ which we write conveniently in the form
\begin{equation}
\label{3eqCond1}
\frac{1}{\gamma} > d \left(\frac{1}{2}-\frac{1}{p+1}\right)\,.
\end{equation} 
Now we deal with the case $d\leqslant 8$ and $d \geqslant 8$ separately.

\bigskip 

\noindent\textit{Case 1.} Assume that $d\leqslant 8$ so that $\sigma (p,d)=\frac{1}{2}$. Recall that by~\eqref{eq.growthY}, for any $\frac{1}{2}$-Schrödinger admissible pair $(q,r)$, and $v(t_n)\in B_{X^{\sigma}}(\lambda \left(\frac{\pi}{4}-t\right)^{-\frac{\alpha(p,d)}{2}})$ we have 
\[
    \|v\|_{L^{q}_{[t_n,t_{n+1}]}L^r}^p \lesssim \left( \frac{\pi}{4} -t \right)^{-K}
\] 
for some $K>0$ which only depends on $d$ and $p$. Indeed such a bound can be obtained for $w=v-w_L$ associated to initial data $v(t_n)$, and the same estimates are obtained for $v_L$ as soon as $r<2+\frac{4}{d}$. 

The norm $L^{p\gamma '}L^{p+1}$ is $\frac{1}{2}$-Schrödinger admissible if and only if 
\begin{equation}
    \label{3eqCond2}
    p\gamma ' >2\,,
\end{equation}
\begin{equation}
\label{3eqCond3}
\frac{1}{p\gamma '} > \frac{d-1}{4}-\frac{d}{2(p+1)}\,,
\end{equation} 
and
\begin{equation}
\label{3eqCond4}
    \frac{d-1}{4}-\frac{d}{2(p+1)}>0\,.
\end{equation}
Condition~\eqref{3eqCond4} is equivalent to $p \geqslant 1+ \frac{2}{d-1}$ which is satisfied, since $p \geqslant 1 + \frac{2}{d}$ by hypothesis. Now observe that once conditions~\eqref{3eqCond1}, \eqref{3eqCond2} and \eqref{3eqCond3} are met, we obtain
\begin{equation}
    \label{3eqTermII}
    II \leqslant \left(\frac{\pi}{4}-t\right)^{-K_1}\tau^{\beta_1}\lambda ^p\,,
\end{equation}
for some $K_1>0$ and $\beta_1>0$. 

We remark that Conditions~\eqref{3eqCond1},~\eqref{3eqCond2} and~\eqref{3eqCond3} are satisfied as soon as 
\begin{equation}
    \left\{
    \begin{array}{ccc}
         p & < & \frac{d+3}{d-3}  \\
         p & \leqslant & \frac{5-d+\sqrt{9d^2-2d+9}}{2(d-1)}\,.
    \end{array}
    \right.
\end{equation}
Observe that $\frac{5-d+\sqrt{9d^2-2d+9}}{2(d-1)} \leqslant \frac{d+3}{d-3}$. We obtain that the conditions~\eqref{3eqCond1}, \eqref{3eqCond2} and~\eqref{3eqCond3} are satisfied for $p<p_{\text{max}}$, which is satisfied by hypothesis. 

\bigskip

\noindent\textit{Case 2.} Assume that $d \geqslant 8$ and $p \geqslant 1+ \frac{3}{d-2}$. With $\sigma = \sigma (p,d)^{-}$, conditions~\eqref{3eqCond1}, \eqref{3eqCond2} and~\eqref{3eqCond3} become:
\begin{equation}
    \left\{
    \begin{array}{ccc}
        \frac{1}{\gamma} & > & d \left(\frac{1}{2}-\frac{1}{p+1}\right) \\
         p\gamma ' & > & 2 \\
         \frac{1}{\gamma ' }& > & \frac{p}{2}\left(\frac{d}{2}-\sigma (p,d)-\frac{d}{p+1}\right)\,,
    \end{array}
    \right.
\end{equation}
which is equivalent to 
\begin{equation}
    \left\{
    \begin{array}{c}
         (d-2)p^2+(d-6)p-2d-4<0 \\
         (d-2)p^3+(d-4)p^2-6p-2d-4<0
         (d-2)p^2+(d-4)p-2d-2>0\,.
    \end{array}
    \right.
\end{equation}
We observe that the second condition is precisely $p<p_{\text{max}}(d)$ and that the first is equivalent to $p \leqslant \frac{d+2}{d-2}=1+\frac{4}{d-2}$ which is satisfied. The last condition is equivalent to $p>\frac{4-d+\sqrt{9d^2-16d}}{2(d-2)}$ which is smaller that $1+\frac{3}{d-2}$, so that the previous conditions are satisfied if and only if $p<p_{\text{max}}(d)$.

We turn to estimating $I$. In order to do so, applying a Sobolev embedding in time and switching derivatives from time to space and provided we fix $\varepsilon >0$ sufficiently small, gives the existence of constants $C\coloneqq C_{\varepsilon}$, $\beta_2 \coloneqq \beta_2 (\varepsilon)$ and an integer $q\coloneqq q(\varepsilon)$ satisfying 
\[
    I \leqslant C\tau^{\beta_2}\|S(\cdot - t_n)v(t_n)\|_{L^{q}_{\langle \cdot \rangle^{-2}\,\mathrm{d}t}(\mathbb{R},\mathcal{W}^{\varepsilon ,p+1})}\,.
\]
A complete proof ot this claim is postponed to Lemma~\ref{3sobolDecoupl} which we refer to for the details. 

We introduce the sets 
\[
    B_{n,\lambda}\coloneqq\left\{u_0 \in X^0, \|S(\cdot -t_n)\phi_{t_n}u_0\|_{L^{q}_{\langle \cdot \rangle^{-2}\,\mathrm{d}t}( \mathbb{R},\mathcal{W}^{\varepsilon ,p+1})}> \lambda \left(\frac{\pi}{4}-t\right)^{-\frac{\alpha (p,d)}{2}}\right\}
\] 
Let us also define
\[
    B'_{n,\delta}:=\{u_0 \in X^0, \|S(\cdot -t_n)u_0\|_{L^{q}_{\langle \cdot \rangle^{-2}\,\mathrm{d}t}(\mathbb{R},\mathcal{W}^{\varepsilon ,p+1})} > \delta\}\,.
\]
An application of Proposition~\ref{3propMeasureEvol} yields
\begin{align*}
    \nu_0(B_{n,\lambda}) & \leqslant \nu_{t_n}\left(\phi_{t_n}B_{n,\lambda}\right)^{\cos(2t_n)^{\alpha (p,d)}} \\
    & \leqslant \mu \left(\phi_{t_n}B_{n,\lambda}\right)^{\cos(2t_n)^{\alpha (p,d)}} \\
    & \leqslant \mu\left(B_{n,\lambda (\frac{\pi}{4}-t_n)^{-\frac{\alpha(p,d)}{2}}}'\right)^{\cos(2t_n)^{\alpha (p,d)}}\,.
\end{align*}
Now we claim that there exists constants $C,c>0$ only depending on $d,p$ such that that 
\begin{equation}
    \label{3eqProbaGaussien}
    \mu (B_{n,\delta}') \leqslant e^{-c\delta^2}\,.
\end{equation}
Assuming~\eqref{3eqProbaGaussien}, we conclude that:
\begin{equation}
    \label{3eqNu0B}
    \nu_0(B_{n,\lambda}) \leqslant e^{-c\lambda^2}\,.
\end{equation}
 
Next, we see that for 
\[
    u_0 \in \left(\phi_{t_n}\right)^{-1}\left(X^0\setminus B_{X^{\sigma}}\left(\lambda (\frac{\pi}{4}-t)^{-\alpha(p,d)/2)}\right)\right) \cap \left(X^0\setminus B_{n,\lambda}\right) \cap \left(X^0 \setminus A_{t_n,\lambda}\right)\,,
\] 
where we recall that $A_{t_n,\lambda}$ is defined by~\eqref{3eqAtlambda}, we have $\phi_{t_n}u_0 \in X^0\setminus B_{X^{\sigma}}\left(\lambda (\frac{\pi}{4}-t)^{-\alpha(p,d)/2)}\right)$ and then:
\begin{equation}
    \label{3eqTermI}
    I \lesssim \tau^{\beta_2} \left(\frac{\pi}{4}-t\right)^{-\frac{\alpha(p,d)}{2}}\lambda\,.
\end{equation}
Taking into account~\eqref{3eqTermI} and~\eqref{3eqTermII} in~\eqref{3eqvtt0} imply that we can adjust $K_2, K_3 >$ such that
\begin{equation}
    \label{3eqTermGlued}
    \tau \sim \lambda^{-K_2} \left(\frac{\pi}{4}-t\right)^{K_3} \text{ and } \|v(t')\|_{L^{p+1}} - \|v(t_n)\|_{L^{p+1}}\leqslant 2\lambda \text{ for } t' \in [t_n,t_{n+1}]\,,
\end{equation}

Let us prove~\eqref{3eqProbaGaussien}, which follows from
\begin{equation}
    \label{3eqLpboundsProba}
    \left\| S(\cdot) f\right \|_{L^{p_0}_{\omega}L^{q}_{\langle \cdot \rangle^{-2}\,\mathrm{d}t}(\mathbb{R},\mathcal{W}^{\varepsilon ,p+1})} \lesssim \sqrt{p_0}\,,
\end{equation}
at least for $p_0 \geqslant \max \{q,r\}$. Assuming such a bound gives the claim using the Markov inequality, and we skip the details, see Appendix~A of~\cite{burqTzvetkov2} for instance. 

It remains to prove~\eqref{3eqLpboundsProba}. Let $p_0\geqslant \max \{q,r\}$. By Minkowski's inequality, Theorem~\ref{3theoremCentral} and the finiteness of the measure $\langle t \rangle ^{-2} \,\mathrm{d}t$, we have
\begin{align*}
    \left\| S(\cdot) f\right \|_{L^{p_0}_{\omega}L^{q}_{\langle t\rangle^{-2} \mathrm{dt}}(\mathbb{R},\mathcal{W}^{\varepsilon,p+1})} &\lesssim \left\| \left\| \sum_{n \geqslant 0} \frac{e^{it \lambda_n^2}}{\lambda _n^{1- \varepsilon}} g_n e_n\right \|_{L^{p_0}_{\omega}}\right\|_{L^{q}_{\langle t\rangle^{-2} \mathrm{dt}}(\mathbb{R},L^{p+1})} \\
    & \lesssim \sqrt{p_0}\|(\lambda_n^{-1+\varepsilon}e_n)\|_{L^{p+1}_x\ell ^2_{\mathbb{N}}}\,.
\end{align*}
Another use of the Minkowski inequality, recalling that $p+1 \in \left(2,\frac{2d}{d-2}\right)$ and Lemma~\ref{3lemmaGainSobolev} finally prove the claim if we choose $\varepsilon >0$ small enough.

To conclude the proof, we repeat the usual argument. The set of \textit{good} initial data $G_{\lambda}$ is defined as the set of initial data $u_0 \in X^0$ which give rise to solutions defined on $[0,t]$ and such that $\|v(t')\|_{L^{p+1}} \leqslant  2\lambda$ for every $t'\leqslant t$. We then observe that the analysis developed to get the bound~\eqref{3eqTermGlued} leads to the inclusion:
\[
    \bigcap_{n=0}^{\lfloor \frac{t}{\tau}\rfloor} \left( \left(\phi_{n\tau}\right)^{-1}(X^0\setminus A^{\circ}_{\lambda(\pi /4-t)^{-\alpha(p,d)/2}}) \cap \left( X^0 \setminus B_{n,\lambda}\right) \cap \left(X^0 \setminus A_{t_n,\lambda}\right)\right) \subset G_{\lambda}\,.
\]
The set of \textit{bad} initial data $B_{\lambda} \coloneqq X^0 \setminus G_{\lambda}$ satisfies 
\[
    \nu_0(B_{\lambda}) \lesssim \lambda^{K_2} \left(\frac{\pi}{4}-t\right)^{-K_3}e^{-c \lambda ^2}\,,
\]  where we used Lemma~\ref{3propMeasureEvol} to bound
\begin{align*}
    \nu_0 \left(\left(\phi_{n\tau}\right)^{-1} \left( X^0\setminus A^{\circ}_{\lambda (\pi /4-t)^{-\alpha(p,d)/2}}\right)\right) & \leqslant \nu_{n\tau} \left(X^0\setminus A^{\circ}_{\lambda (\pi /4-t)^{-\alpha(p,d)/2}}\right)^{\cos (2t_n)^{\alpha (p,d)}} \\
    & \leqslant e^{-c \lambda ^2}\,,
\end{align*}
and the expression of $\tau$. Finally the measure of this set is made smaller than $\eta$ by taking $\lambda \sim \log^{\frac{1}{2}}\left(\frac{1}{\eta}\right)  \left\vert \log \left(\frac{\pi}{4}-|t|\right)\right \vert ^{\frac{1}{2}}$, which ends the proof.

Finally we treat the case $p< 1+\frac{2}{d-1}$, which is much simpler. For $\sigma <\frac{1}{2}$ close enough to $\frac{1}{2}$ we have the Sobolev embedding $\mathcal{H}^{\sigma} \hookrightarrow L^{p+1}$, so that $\|v(t')-v(t_n)\|_{L^{p+1}} \les \|v(t')-v(t_n)\|_{\mathcal{H}^{\sigma}}$. Writing $v(t')=e^{i(t'-t_n)}v(t_n) + w(t')$, we obtain the bound: 
\begin{align*}
     \|v(t')\|_{L^{p+1}} \leqslant & \|v(t_n)\|_{L^{p+1}} + \|v(t)-v(t_n)\|_{\mathcal{H}^{\sigma}} \\
     & \leqslant \|v(t_n)\|_{L^{p+1}} + 2\|v(t_n)\|_{\mathcal{H}^{\sigma}} + \|w(t')\|_{\mathcal{H}^{\sigma}}\,.
\end{align*}
Then, one can use Proposition~\ref{3propLocal} to control $\|w(t')\|_{\mathcal{H}^{\sigma}}$ and use a similar argument as above to conclude. 
\end{proof}

\subsection{Scattering for (NLS) and end of the proof}\label{3secEnd}

In order to prove Theorem~\ref{3main} we need a scattering result for~\eqref{3NLSH} in $\mathcal{H}^{\varepsilon}$. The main ingredient in the proof is the following proposition. 

\begin{proposition}\label{3propMain} Let $d\in\{2, \dots, 10\}$ and $p \in \left(1+\frac{2}{d}, 1+\frac{4}{d}\right)$. For almost every $u_0 \in X$ the unique global associated solution to~\eqref{3NLSH} constructed by Corollary~\ref{3coro1} and Corollary~\ref{3coro2} satisfies the following estimates.
\begin{enumerate}[label=(\textit{\roman*})]
    \item There exist $\sigma, \delta >0$ only depending on $p$ and $d$, and $u_{\pm} \in \mathcal{H}^{-\sigma}$ such that 
    \begin{equation}
        \label{3negativeScattering}
        \|v(t)-e^{-itH}(u_0-u_{\pm})\|_{\mathcal{H}^{-\sigma}} \lesssim C(u_0)\left(\frac{\pi}{4}-t\right)^{\delta} \underset{t \to \pm\frac{\pi}{4}}{\longrightarrow} 0\,.
    \end{equation}
    \item There exists $\varepsilon >0$ such that for all $t \in \left(-\frac{\pi}{4},\frac{\pi}{4}\right)$:
    \begin{equation}
        \label{3positiveBound}
        \|v(t)\|_{\mathcal{H}^{\varepsilon}} \lesssim_{\varepsilon} C(u_0)\,.
    \end{equation}
\end{enumerate}
In both cases there exist numerical constants $c,C>0$ such that $\mu (C(u_0)>\lambda) \leqslant C e^{-c\lambda}$. 
\end{proposition}

\begin{proof} Again we present the proof for the forward scattering part. We work with initial data in a set of full measure $G$ which may be taken as the intersection of the full measure sets constructed in Corollary~\ref{3coro1} and Corollary~\ref{3coro2}. Then every $u_0 \in G$ gives rise to a global solution to~\eqref{3NLSH}, 
\[v(t)=u_L(t)+w(t)=e^{-itH}u_0+w(t)\]  
satisfying the bounds~\eqref{eq:boundGrowthPolynomial} and~\eqref{3eqLpwithoutSobolev}. Up to taking the intersection with another set of full measure we can always assume that $\|u_L\|_{L^q\mathcal{W}^{s,r}} < \infty$ as long as $(q,r)$ satisfies the hypothesis of Lemma~\ref{3lemmaGainSobolev}. More precisely, we have
\[
    \mu (u_0, \; \|u_L\|_{L^q\mathcal{W}^{s,r}} >\lambda) \leqslant Ce^{-c\lambda^2}\,.
\] 

We start the proof with the case $p<p_{\text{max}}(d)$ and we will use Corollary~\ref{3coro2}. To start the proof we write 
\begin{equation}
    \label{3eqv}
    e^{itH}w(t)=-i\int_0^t e^{isH}\left(\cos(2s)^{-\alpha(p,d)}|u_L(s)+w(s)|^{p-1}(u_L(s)+w(s))\right)\,\mathrm{d}s\,.
\end{equation}
By the Sobolev embedding, let us set $\sigma >0$ such that $L^{\frac{p+1}{p}} \hookrightarrow \mathcal{H}^{- \sigma}$, one can choose for example $\sigma := d \left(\frac{1}{2}-\frac{1}{p+1}\right)$. We claim that there exists $\delta >0$ such that
\begin{equation}
    \label{3weakScat}
    \int_t^{\frac{\pi}{4}}|\cos (2s)|^{-\alpha(p,d)}\|(u_L(s)+w(s))|u_L(s)+w(s)|^{p-1}\|_{\mathcal{H}^{-\sigma}}\,\mathrm{d}s \lesssim \left(\frac{\pi}{4}-t\right)^{\delta}\,.
\end{equation}
Assuming~\eqref{3weakScat} proves that the integral in~\eqref{3eqv} converges absolutely in $\mathcal{H}^{-\sigma}$ and thus convergent to some $u_+ \in \mathcal{H}^{-\sigma}$ and proves~\eqref{3negativeScattering}. 

In order to prove~\eqref{3weakScat} we use the Sobolev embedding and the bound $\cos(2s) \gtrsim \left(\frac{\pi}{4}-s\right)$ to get 
\begin{align*}
    \int_t^{\frac{\pi}{4}}|\cos (2s)|&^{-\alpha(p,d)}\|(u_L(s)+w(s))|u_L(s)+w(s)|^{p-1}\|_{\mathcal{H}^{-\sigma}}\,\mathrm{d}s   \\
    & \lesssim \int_t^{\frac{\pi}{4}}|\cos (2s)|^{-\alpha(p,d)}\|u_L(s)+w(s)\|_{L^{p+1}}^p\,\mathrm{d}s\\
    & \lesssim \int_t^{\frac{\pi}{4}}\left(\frac{\pi}{4}-s\right)^{-\alpha(p,d)}\|u_L(s)\|^{p}_{L^{p+1}} \,\mathrm{d}s + \int_t^{\frac{\pi}{4}}\left(\frac{\pi}{4}-s\right)^{-\alpha(p,d)}\|w(s)\|^{p}_{L^{p+1}}\,\mathrm{d}s \\
    &\lesssim \|u_L\|_{L^{pr}L^{p+1}}^p\left(\int_t^{\frac{\pi}{4}}\left(\frac{\pi}{4}-t\right)^{r'\left(-\alpha(p,d)\right)} \,\mathrm{d}s\right)^{\frac{1}{r'}} \\ 
    &+ C(u_0)^p\int_t^{\frac{\pi}{4}}\left(\frac{\pi}{4}-t\right)^{-\alpha(p,d)} \left\vert \log \left(\frac{\pi}{4}-t\right)\right\vert^{\frac{p}{2}} \,\mathrm{d}s
\end{align*}
where in the last line we used Hölder's inequality with $\frac{1}{r}+\frac{1}{r'}=1$ and the bounds~\eqref{3eqLpwithoutSobolev}. As explained above, without loss of generality we can assume $\|u_L\|_{L^{pr}L^{p+1}}$ to be finite. It remains to choose $r'$ such that $r'\alpha(p,d)<1$ which ensures that the previous integrals are absolutely convergent and that they can be bounded by a positive power of $\frac{\pi}{4}-t$. This gives~\eqref{3weakScat}. 

It remains to prove~\eqref{3positiveBound}. First observe that $w$ satisfies 
\[
    i\partial _t w(t) -Hw(t)=\cos (2t)^{-\alpha(p,d)}(u_L(t)+w(t))|u_L(t)+w(t)|^{p-1}\,
\] 
then apply $H^{\frac{\varepsilon}{2}}$, multiply by $\overline{H^{\frac{\varepsilon}{2}}w(t)}$ and integrate in space to obtain:
\begin{align*}
    \frac{\mathrm{d}}{\mathrm{d}t} \left(\|w(t)\|^2_{\mathcal{H}^{\varepsilon}}\right)
    & = 2\cos(2t)^{-\alpha(p,d)} \operatorname{Im} \left( \int_{\mathbb{R}^2} H^{\frac{\varepsilon}{2}}\left(|u_L(t)+w(t)|^{p-1}(u_L(t)+w(t))\right) \overline{H^{\frac{\varepsilon}{2}}w(t)}\,\mathrm{d}x\right) \\
    &\lesssim \cos(2t)^{-\alpha(p,d)} \|(u_L(t)+w(t))|u_L(t)+w(t)|^{p-1}\|_{\mathcal{W}^{\varepsilon, \frac{p+1}{p}}}\|w(t)\|_{\mathcal{W}^{\varepsilon, p+1}}\,,
\end{align*}
where we used the Hölder inequality. By Sobolev's product laws and the minoration of the $\cos$ function we obtain
\begin{align*}
    \frac{\mathrm{d}}{\mathrm{d}t} \left(\|w(t)\|^2_{\mathcal{H}^{\varepsilon}}\right) &\lesssim \left(\frac{\pi}{4}-t\right)^{-\alpha(p,d)} \left(\|u_L(t)\|^{p-1}_{L^{p+1}} + \|w(t)\|^{p-1}_{L^{p+1}}\right) \\
    & \times \left(\|u_L(t)\|_{\mathcal{W}^{\varepsilon, p+1}}+\|w(t)\|_{\mathcal{W}^{\varepsilon, p+1}}\right)\|w(t)\|_{\mathcal{W}^{\varepsilon, p+1}}\,.
\end{align*}
Then we bound the $L^{p+1}$ norms by the $\mathcal{W}^{\varepsilon,p+1}$ norms and also use $w(t)=v(t)-u_L(t)$, which leads to
\[
    \frac{\mathrm{d}}{\mathrm{d}t} \left(\|w(t)\|^2_{\mathcal{H}^{\varepsilon}}\right) \lesssim \left(\frac{\pi}{4}-t\right)^{-\alpha(p,d)} \left(\|u_L(t)\|^{p+1}_{\mathcal{W}^{\varepsilon ,p+1}}+ \|v(t)\|^{p+1}_{\mathcal{W}^{\varepsilon ,p+1}}\right)\,.
\]
After integration we have
\begin{align*}
    \|w(t)\|^2_{\mathcal{H}^{\varepsilon}} & \leqslant \underbrace{\|w_0\|^2_{\mathcal{H}^{\varepsilon}}}_{=0} + C\int_0^t \left(\frac{\pi}{4}-s\right)^{-\alpha(p,d)} \left(\|u_L(t)\|^{p+1}_{\mathcal{W}^{\varepsilon ,p+1}}+ \|v(t)\|^{p+1}_{\mathcal{W}^{\varepsilon ,p+1}}\right)\,\mathrm{d}s\\
    &\lesssim \|u_L\|_{L^{(p+1)r}\mathcal{W}^{\varepsilon,p+1}}^{p+1}\left(\int_0^{\frac{\pi}{4}}\left(\frac{\pi}{4}-t\right)^{r'\left(-\alpha(p,d)\right)} \,\mathrm{d}s \right)^{\frac{1}{r'}} \\
    &+ C(u_0)^{p+1}\int_0^{\frac{\pi}{4}}\left(\frac{\pi}{4}-t\right)^{-\alpha(p,d)}\left\vert \log \left(\frac{\pi}{4}-t\right)\right\vert ^{\frac{p+1}{2}} \,\mathrm{d}s\,,
\end{align*}
where we used Hölder's inequality with $\frac{1}{r}+\frac{1}{r'}=1$ and~\eqref{3eqLpwithoutSobolev}. As above without loss of generality we can assume $\|u_L\|_{L^{pr}\mathcal{W}^{\varepsilon,p+1}}< \infty$. Then choosing $r'$ such that $r'\alpha(p,d)<1$ all the above integrals are convergent, which proves~\eqref{3positiveBound}.

Now we explain the case $p \geqslant p_{\text{max}}(d)$. As remarked before this case is only necessary when $d \geqslant 8$. Let $p \in \left[p_{\text{max}}(d),1+\frac{4}{d}\right)$. In order to prove~\eqref{3negativeScattering} we need to prove that there exists $\delta >0$ such that~\eqref{3weakScat} holds. Recalling the bounds on $u_L$ one only needs to find $\delta >0$ such that 
\begin{equation}
    \label{3weakScat2}
    \int_t^{\frac{\pi}{4}}|\cos (2s)|^{-\alpha(p,d)}\|w(s)\|_{L^{p+1}}^{p}\,\mathrm{d}s \lesssim \left(\frac{\pi}{4}-t\right)^{\delta}\,.
\end{equation}
Similarly, in order to prove~\eqref{3positiveBound}, an examination of the above proof shows that one only needs to prove that for sufficiently small $\varepsilon >0$, the integral
\begin{equation}
    \label{3weakScat3}
    \int_0^{\frac{\pi}{4}}|\cos (2s)|^{-\alpha(p,d)}\|w(s)\|_{\mathcal{W}^{\varepsilon, p+1}}^{p+1}\,\mathrm{d}s
\end{equation}
is finite. The proof of~\eqref{3weakScat2} and~\eqref{3weakScat3} are essentially the same, thus the proof of~\eqref{3weakScat2} is carried out in details and we only explain the modifications for~\eqref{3weakScat3}. We know, thanks to Corollary~\ref{3coro1} that there is a constant $C>0$ such that
\[
    \|w(s)\|_{\mathcal{H}^{\sigma}} \leqslant C \left(\frac{\pi}{4}-s\right)^{-\frac{\alpha (p,d)}{2} - \varepsilon_0}\,,
\]
for all $s>0$, and where $\varepsilon_0 >0$ is arbitrarily small. Therefore using the notation of Lemma~\ref{3lemmaContraction} the local well-posedness time $\tau(t_0)$ at time $t_0$ has the form $\tau(t_0) \sim \left(\frac{\pi}{4}-t_0\right)^{\frac{(p+1)\alpha q_1}{2}}$, and we can verify that we can take $q_1$ such that $\frac{(p+1)\alpha q_1}{2}=1$, which can be satisfied as soon as 
\[
    \frac{(p+1)\alpha}{2}< 1-\left(\frac{d}{4}-\frac{\sigma}{2}\right)(p-1)\,,
\]
which is the case for $p\geqslant \frac{-d+4+\sqrt{d^2+32}}{4}$. One can check that $P_d(\frac{-d+4+\sqrt{d^2+32}}{4})<0$ for $d\in\{8, 9, 10\}$, meaning that for $p\geqslant p_{\text{max}}(d)$ we can set $\tau (t_0)\sim \left(\frac{\pi}{4}-t_0\right)$. 

We can consider a sequence of times $t_n =\frac{\pi}{4}(1-2^{-n})$ such that the $[t_n,t_{n+1}]$ are local well-posedness intervals, on which we have the bound 
\begin{equation}
    \label{eq:boundLocalInterval}
    \|w\|_{Y^{\sigma(p,d)^-}_{[t_n,t_{n+1}]}} \leqslant \left(\frac{\pi}{4}-t_n\right)^{-\frac{\alpha}{2}-\varepsilon_0}\,.
\end{equation}

We claim that the real sequence $(u_n)_{n\geqslant 0}$ defined by 
\[
    u_n := \int_0^{t_n}|\cos (2s)|^{-\alpha(p,d)}\|w(s)\|_{L^{p+1}}^{p} \,\mathrm{d}s
\] 
is a Cauchy sequence. In order to prove this claim, we use the bound $\cos (2s) \gtrsim \left(\frac{\pi}{4}-s\right)$ and Hölder's inequality with $\frac{1}{\gamma}+\frac{1}{\gamma'}=1$, where $\gamma$ such that there exists $\sigma \in [0,\sigma (p,d))$ such that $(p\gamma',p+1)$ is ($\sigma(p,d)$)-Schrödinger admissible, that is $p\gamma ' \geqslant 2$ and $\frac{2}{p\gamma '}+\frac{d}{p+1}=\frac{d}{2}-\sigma$. Note that the condition $p\gamma ' \geqslant 2$ can be written as  $\frac{d}{2}-\frac{d}{p+1}-\sigma (p,d) <1$ which is equivalent to $(d-2)p^2+(d-6)p-2d-4<0$, that is $p< \frac{d+2}{d-2}$. This last inequality is always satisfied in our case. We also mention that the $\sigma (p,d)$ admissibility also requires that $\frac{d}{2}-\frac{d}{p+1}-\sigma (p,d)>0$, which is equivalent to $p>\frac{1}{-d+4+\sqrt{9d^2-16d}}{2(d-2)}$. Since $p > 1+\frac{3}{d-2}>\frac{1}{-d+4+\sqrt{9d^2-16d}}{2(d-2)}$, this last inequality is, again, always satisfied. 

We can summarise the above discussion: provided $-\alpha (p,d) + \frac{1}{\gamma}>0$, which will be checked later, we obtain:
\begin{align*}
    |u_{n+1}-u_n| & \leqslant \int_{t_n}^{t_{n+1}}\left(\frac{\pi}{4}-s\right)^{-\alpha (p,d)} \|w(s)\|_{L^{p+1}}^p \,\mathrm{d}s\\
    &\lesssim \left(\frac{\pi}{4}-t_n\right)^{-\alpha (p,d)+ \frac{1}{\gamma}} \|w\|_{L^{p\gamma '}_{[t_n,t_{n+1}]}L^{p+1}}^p \\
    &\lesssim \left(\frac{\pi}{4}-t_n\right)^{-\alpha (p,d)+ \frac{1}{\gamma}} \|w\|_{Y^{\sigma(d,p)^{-}}_{[t_n,t_{n+1}]}}^p\,.
\end{align*}
Using~\eqref{eq:boundLocalInterval} we infer
\begin{align*}
    |u_{n+1}-u_n| &\lesssim \left(\frac{\pi}{4}-t_n\right)^{-\alpha (p,d)+\frac{1}{\gamma}} \left(\frac{\pi}{4}-t_{n+1}\right)^{-\frac{p(\alpha(p,d)+\varepsilon_0)}{2}} \\
    &\lesssim 2^{-n(\delta+\varepsilon_0)} \,,
\end{align*}
with 
\begin{equation}\label{eq.delta}
\delta \coloneqq -\frac{p+2}{2}\alpha (p,d)+\frac{1}{\gamma}    
\end{equation}
and $\varepsilon_0$ arbitrarily small. Assume that $\delta >0$, then we conclude that $(u_n)_{n \geqslant 0}$ is a Cauchy sequence and moreover if we choose $n$ such that $\frac{\pi}{4}-t \in [2^{-(n+1)},2^{-n}]$ then we can bound: 
\begin{align*}
    \int_t^{\frac{\pi}{4}}|\cos (2s)|^{-\alpha(p,d)}\|w(s)\|_{L^{p+1}}^{p}\,\mathrm{d}s & \lesssim  \sum_{k \geqslant n} 2^{-k\delta} \lesssim 2^{-n \delta} \lesssim \left(\frac{\pi}{4}-t\right)^{\delta} \,.
\end{align*}
This proves~\eqref{3weakScat2}, provided we show that $\delta >0$. Before checking this fact we run the same analysis for~\eqref{3weakScat3} in order to obtain the inequalities that the parameters must satisfy. This time we set 
\[
    z_n := \int_0^{t_n}|\cos(2s)|^{-\alpha(p,d)}\|w(s)\|_{\mathcal{W}^{\varepsilon,p+1}}^{p+1} \,\mathrm{d}s\,,
\]
and we only need to prove that this sequence forms a Cauchy sequence. Again we use Hölder's inequality with $\frac{1}{\gamma}+\frac{1}{\gamma '}=1$ such that there exists $\sigma \in [0,\sigma (p,d)$ such that $(\gamma '(p+1),p+1)$ is $\sigma(p,d)$-Schrödinger admissible, that is $\gamma '(p+1)\geqslant 2$ and $\frac{2}{(p+1)\gamma'}+\frac{d}{p+1}=\frac{d}{2}-\sigma $.
The condition $\gamma '(p+1)\geqslant 2$ is equivalent to $p \leqslant \frac{d+2}{d-2}$ as above, which is satisfied. Then, with the same estimates as above and provided $-\alpha (p,d) + \frac{1}{\gamma}>0$ we arrive at
\[
    |z_{n+1}-z_n| \les \left(\frac{\pi}{4}-t_n\right)^{-\alpha (p,d)+\frac{1}{\gamma}}\|v\|^{p+1}_{Y^{\sigma(p,d)^{-}}_{[t_n,t_{n+1}]}}\,.
\]
Using~\eqref{eq:boundLocalInterval} and the fact that $\left(\frac{\pi}{4} -t_n\right)\sim C2^{-n}$, in order to get 
\[
    |z_{n+1}-z_n| \les 2^{-n(\tilde{\delta}+\varepsilon_0)}
\] 
where
\begin{equation}\label{eq.deltaTilde}
\tilde{\delta} \coloneqq -\displaystyle\frac{(p+3)\alpha (p,d)}{2}+\frac{1}{\gamma}\,.    
\end{equation}
Assuming that $\tilde{\delta} >0$ proves that $(u_n)_{n \geqslant 0}$ is a Cauchy sequence, and ends the proof of~\eqref{3weakScat3} 

We need to find the range of $p$ (depending on $d$) for which $\delta, \tilde{\delta}  >0$. Note that would immediately imply $\alpha (p,d)-\frac{1}{\gamma}>0$ which was a necessary condition to bound $|u_{n+1}-u_n|$ and $|z_{n+1}-z_n|$.

From~\eqref{eq.delta}, the claim $\delta >0$ is equivalent to:
\[
    \frac{p+2}{2} \left(-\alpha(p,d)\right)+1-\frac{p}{2}\left(\frac{d}{2}-\sigma (p,d)-\frac{d}{p+1}\right)>0\,,\] which after simplification reads \[Q_d(p):=2p^3+dp^2+(d-6)p-2d-4>0\,.
\]
Then we can compute that, as a polynomial in $p$, $\operatorname{discrim}(Q_d)=9d^4-76d^3+36d^2-1728d$ thus $Q_d$ has exactly one real root if $d\leqslant 9$ and exactly three for $d\geqslant 10$. In any cases we check that $Q_d$ has exactly one positive real root for $d\geqslant 8$. Let us denote by $p_d$ this root and we claim that $p_{\text{max}}(d)>p_d$. To verify this statement we remark that both $P_d$ and $Q_d$ are increasing $[1,1+\frac{4}{d}]$ thus we only need to check that $P_d \leqslant Q_d$ in that region, which is equivalent to $(d-4)p^2-4p-d<0$ which is satisfied for $p \geqslant \frac{d}{d-4}$ ( which is greater than $1+\frac{4}{d})$) and $d \geqslant 8$. 

In a similar fashion, from ~\eqref{eq.deltaTilde}, the claim $\tilde{\delta}>0$ is equivalent to: 
\[
    \frac{p+3}{2} \left(-\alpha(p,d)\right)-\frac{p+1}{2}\left(\frac{d}{2}-\sigma(p,d)-\frac{d}{p+1}\right) >0\,,
\] which after simplification reads 
\[
    R_d(p)\coloneqq p^2+\frac{d}{2}p-\frac{d}{2}-3>0\,,
\] 
that is $p>-\frac{d}{4}+\frac{1}{4}\sqrt{d^2+8d+48}$. We claim that 
\[
    p_{\text{max}}(d)>-\frac{d}{4}+\frac{1}{4}\sqrt{d^2+8d+48}\coloneqq p_d'
\] if and only if $d\leqslant 24$. This claim is equivalent to $P_{24}(p_{24}')P_{25}(p_{25}')<0$. This reduces to $(813\sqrt{51}-5806)(44181\sqrt{97}-435131)<0$ which is the case. 
\end{proof}

\subsection{Proof of the main theorems}

\begin{proof}[Proof of Theorem~\ref{3side}] The global well-posedness part and the estimate of the $H^{\sigma}$ norm follows from Corollary~\ref{3coro1} and Lemma~\ref{3lens}. The estimate of the $L^{p+1}$ norm is a consequence of Corollary~\ref{3coro2} and Lemma~\ref{3lens}. 
\end{proof}

\begin{proof}[Proof of Theorem~\ref{3main}] 
In order to end the proof of the theorem we assume that there exist $\delta, \sigma, \varepsilon >0$ such as constructed in Proposition~\ref{3propMain}. This immediately proves part (\textit{i}) of Theorem~\ref{3main}.   

We deduce that $u_+ \in \mathcal{H}^{\varepsilon}$. In fact, $e^{itH}w(t)$ being bounded in the Hilbert space $\mathcal{H}^{\varepsilon}$ we can extract a subsequence, weakly converging to some $\tilde{u}_+ \in \mathcal{H}^{\varepsilon}$ but this convergence also holds in $\mathcal{H}^{-\sigma}$ where $e^{itH}w(t) \to u_+$ thus by uniqueness of the weak limit, $u_+=\tilde{u}_+ \in \mathcal{H}^{\varepsilon}$. 

We claim that the bound
\begin{equation}
\label{3scatH}
\|e^{itH}w(t)-u_+\|_{\mathcal{H}^{\varepsilon_0}} \lesssim \left(\frac{\pi}{4}-t\right)^{2\varepsilon_0}
\end{equation}
implies the estimate~\eqref{3scat1}. Indeed, let $u$ be the solution to Schrödinger equation~\eqref{3NLS} associated to $u_0$.  We denote by $s=\frac{\tan t}{2}$ the time variable of $u$ where $t$ is the time variable of $v:= \mathcal{L}u$. We refer to Appendix~\ref{3appendixLens} for details. Then from Lemma~\ref{3lens} we have 
\begin{align*}
    \|u(s)-e^{is\Delta_y}(u_0+u_+)\|_{\mathcal{H}^{\varepsilon_0}} &\lesssim \left(\frac{\pi}{4}-t(s)\right)^{-\varepsilon_0} \|\mathcal{L}u-\mathcal{L}(e^{is\Delta_y}(u_0+u_+))\|_{\mathcal{H}^{\varepsilon_0}} \\
    &\lesssim \left(\frac{\pi}{4}-t(s)\right)^{-\varepsilon_0} \|w(t(s))-e^{-it(s)H}u_+\|_{\mathcal{H}^{\varepsilon_0}}\\
    &\lesssim \left(\frac{\pi}{4}-t(s)\right)^{-\varepsilon_0} \|e^{it(s)H}w(t(s))-u_+\|_{\mathcal{H}^{\varepsilon_0}}\\
    & \underset{s \to \infty}{\longrightarrow} 0\,,
\end{align*}
where we used that $v(t(s))=w(t(s))+e^{-it(s)H}u_0$ and~\eqref{3scatH}. 

In order to prove~\eqref{3scatH}, let $\theta \in [0,1]$ and introduce $\sigma(\theta):= -\sigma \theta+(1-\theta)\varepsilon$. Interpolating between~\eqref{3negativeScattering} and~\eqref{3positiveBound} we have
\begin{align*}
    \|e^{itH}w(t)-u_+\|_{\mathcal{H}^{\sigma (\theta)}} & \leqslant \|e^{itH}w(t)-u_+\|_{\mathcal{H}^{\varepsilon}}^{1-\theta} \|e^{itH}w(t)-u_+\|^{\theta}_{\mathcal{H}^{-\sigma}} \\
    & \lesssim C(\varepsilon)^{1-\theta} \|e^{itH}w(t)-u_+\|^{\theta}_{\mathcal{H}^{-\sigma}} \\
    & \lesssim \left(\frac{\pi}{4}-t\right)^{\theta \delta}\,.
\end{align*}
We claim that we can find $\varepsilon_0 \in(0,\varepsilon)$ satisfying $\sigma (\theta)=\varepsilon _0$ and $\delta \theta = 2\varepsilon _0$. Indeed one can take 
\[
    \varepsilon _0 \coloneqq  \frac{\varepsilon}{1+\frac{2}{\delta}\left(\sigma + \varepsilon\right)}\,.
\] 

Finally in order to prove~\eqref{3scat2}, observe that by properties of the lens transform we have: 
\begin{align*}
    \|e^{-is\Delta_y}u(s)-(u_0+u_+)\|_{\mathcal{H}^{\varepsilon_0}} & = \|e^{it(s)H}(u(s) - e^{is\Delta_y}(u_0+u_+))\|_{\mathcal{H}^{\varepsilon_0}} \\
    & \lesssim \left(\frac{\pi}{4}-t(s)\right)^{-\varepsilon_0}\|u(s) - e^{is\Delta_y}(u_0+u_+))\|_{\mathcal{H}^{\varepsilon_0}}
\end{align*}
which converges to zero thanks to~\eqref{3scat1}, up to taking $\varepsilon >0$ small enough. 

We have thus proven the convergences~\eqref{3scat1} and~\eqref{3scat2} at a rate which is $\left(\frac{\pi}{4}-t(s)\right)^{\varepsilon _0}$ as $s \to \infty$. To conclude it suffices to remark that:
\[\frac{\pi}{4}-t(s)=\frac{1}{2}\operatorname{arctan} \left(\frac{1}{2s}\right) \sim \frac{C}{s} \text{ as } s\to \infty\qedhere\]
\end{proof}

\appendix

\section{Technical estimates in harmonic Sobolev spaces}\label{3appTechnical}

In this appendix we recall some well-known facts concerning the harmonic oscillator-based Sobolev spaces. 

\begin{lemma}[Harmonic Sobolev Spaces]\label{3technicalTools} The following properties hold. 
\begin{enumerate}[label=(\textit{\roman*})]
    \item For $\sigma \in \mathbb{R}_+$ and $p\in (1,\infty)$, $\|u\|_{\mathcal{W}^{\sigma , p}}=\|H^{\frac{\sigma}{2}}u\|_{L^p} \sim \|\langle \nabla\rangle^{\sigma} u\|_{L^p} + \|\langle x\rangle ^{\sigma} u\|_{L^p}$.
    \item Sobolev embedding: let $p_1, p_2 \geqslant 1$, then in dimension $d$, $\mathcal{W}^{\sigma_1, p_1} \hookrightarrow \mathcal{W}^{\sigma_2, p_2}$ as soon as 
    \[ 
        \frac{1}{p_1} - \frac{\sigma _1}{d} \leqslant \frac{1}{p_2}-\frac{\sigma _2}{d}\,.
    \] 
    \item For $\sigma \geqslant 0$, $q \in (1,\infty)$ and $q_1,q_2,q'_1q'_2 \in (1,\infty]$ one has \[ \|uv\|_{\mathcal{W}^{\sigma ,q}} \lesssim \|u\|_{L^{q_1}}\|v\|_{\mathcal{W}^{\sigma ,q_1'}}+\|u\|_{\mathcal{W}^{\sigma ,q_2}}\|v\|_{L^{q_2'}}\,,\] as soon as $\displaystyle \frac{1}{q}=\frac{1}{q_1}+\frac{1}{q'_1}=\frac{1}{q_2}+\frac{1}{q'_2}$.
    \item (Chain Rule) For $s \in (0,1)$, $p \in (1,\infty)$ and a function $F\in \mathcal{C}^1(\mathbb{R})$ such that $F(0)=0$ and such that there is a $\mu \in L^1([0,1])$ such that for every $\theta \in [0,1]$: \[|F'(\theta v +(1-\theta)w)| \leqslant \mu(\theta) \left( G(v)+G(w) \right)\] where $G>0$; we have 
    \begin{equation}
        \label{3chainRule}
        \|F \circ u \|_{\mathcal{W}^{s,p}} \lesssim \|u\|_{\mathcal{W}^{s,p_0}}\|G \circ u\|_{L^{p_1}}\,.
    \end{equation}
\end{enumerate}
\end{lemma}

\begin{proof}
(\textit{i}) is proved in \cite{burqThomannTzvetkov}. The other statements are proven for usual Sobolev spaces in~\cite{taylor} and their readaptation to harmonic Sobolev spaces results from the use of (\textit{i}). We also refer to~\cite{deng}. For usual Sobolev spaces, (\textit{iii}) and (\textit{iv}) can be found in~\cite{taylor}, Chapter~2. 
\end{proof}

Our next lemma is a technical estimate which aims at decoupling the norm in time. 

\begin{lemma}\label{3sobolDecoupl} Let $\varepsilon >0$, $\alpha \in (0,1)$ and $q, r\geqslant 1$ be such that $W^{\frac{\varepsilon}{2},q}(\mathbb{R}) \hookrightarrow C^{0,\alpha}(\mathbb{R})$. Let also $t_0\in[0,\frac{\pi}{4}]$. Then for every $f \in \mathcal{W}^{\varepsilon,r}$ there holds  \[\|(e^{-i(t-t_0)H}-\operatorname{id})f\|_{L^{\infty}_{[t_0,t_0+\tau]}L^r} \lesssim \tau^{\alpha} \|e^{itH}f\|_{L^{q}_{\langle t\rangle^{-2}\,\mathrm{d}t}( \mathbb{R},\mathcal{W}^{\varepsilon ,r})}\,.\]
where the implicit constant depends on $\varepsilon$, $q$ and $\alpha$. 
\end{lemma}

\begin{proof} Let $\chi (t)$ a smooth function such that for $|t|\leqslant 2\pi$, $\chi (t)=\langle \pi \rangle^{-2}$ and for $|t|\geqslant 2 \pi$ one has $\chi (t) = \langle t \rangle^{-2}$. Set $F(t):=e^{-i(t-t_0)H}f$ for convenience. Then we use the definition of the $C^{0,\alpha}$ norm:
\begin{align*}
    \|(e^{-i(t-t_0)H}-\operatorname{id})f\|_{L^{\infty}_{[t_0,t_0+\tau]}L^r} & \leqslant |t-t_0|^{\alpha}\|F\|_{C^{0,\alpha}([t_0,t_0+\tau],L^r)} \\
    & \leqslant \tau^{\alpha}\|F\|_{C^{0,\alpha}([-\pi,\pi],L^r)}
    \,.
\end{align*}
Now we use that $\|F\|_{C^{0,\alpha}([-\pi,\pi],L^r)} \leqslant C\|\chi(\cdot-t_0)F\|_{C^{0,\alpha}(\mathbb{R},L^r)}$ with a constant $C$ which does not depend on $\tau$. Combined with the Sobolev embedding $W^{\frac{\varepsilon}{2},q}(\mathbb{R}) \hookrightarrow C^{0,\alpha}(\mathbb{R})$ we have: 
\begin{align*}
  \|(e^{-i(t-t_0)H}-\operatorname{id})f\|_{L^{\infty}_{[t_0,t_0+\tau]}L^r} &\lesssim \tau^{\alpha} \|\chi (\cdot - t_0) F\|_{W^{\frac{\varepsilon}{2},q}(\mathbb{R},L^r)}\\
  & \lesssim \tau^{\alpha}\|\chi (t) e^{-itH}f\|_{W^{\frac{\varepsilon}{2},q}(\mathbb{R},L^r)} \,.
\end{align*}
Now observe that in order to transfer derivatives from time to space we need to commute $\chi$ and~$\langle D_t\rangle ^{\frac{\varepsilon}{2}}$, which follows from the following observation:
\begin{align*}
    \langle D_t\rangle ^{\frac{\varepsilon}{2}} \chi &= \left( 1+\left[\langle D_t \rangle ^{\frac{\varepsilon}{2}},\chi\right] \langle D_t \rangle ^{\frac{-\varepsilon}{2}} \chi ^{-1}\right)\chi \langle D_t \rangle ^{\frac{\varepsilon}{2}}\\
    &=(\operatorname{id}+A)\chi \langle D_t \rangle ^{\frac{\varepsilon}{2}}\,,
\end{align*}
with $A:=\left[\langle D_t \rangle ^{\frac{\varepsilon}{2}},\chi\right] \langle D_t \rangle ^{-\frac{\varepsilon}{2}} \chi ^{-1}$. Since $\chi, \chi^{-1}$ are zero order pseudodifferential operators and that $\langle D_t \rangle ^{\pm\frac{\varepsilon}{2}}$ are pseudodifferential operators of order $\pm\frac{\varepsilon}{2}$, the usual pseudodifferential calculus implies that $A$ is of order $-1$ which is regularising and finally $(\text{id}+A)$ is of order zero, thus continuous on all $L^p$ spaces, as soon as $p\in (1,\infty)$, see~\cite{hormander} for instance, and the continuity constant does not depend on $\tau$. This gives 
\begin{align*}
    \|e^{-itH}f\|_{W^{\frac{\varepsilon}{2},q}(\mathbb{R},L^r)} & \lesssim \|\langle D_t \rangle ^{\frac{\varepsilon}{2}}\chi (t)e^{-itH}f\|_{L^{q}(\mathbb{R},L^r)}\\
    & \lesssim \|\chi (t) \langle D_t \rangle ^{\frac{\varepsilon}{2}} e^{-itH}f\|_{L^{q}(\mathbb{R},L^r)}\\
    & \leqslant \|\langle D_t \rangle ^{\frac{\varepsilon}{2}} e^{-itH}f\|_{L^{q}_{\langle t \rangle ^{-2}\,\mathrm{d}t}(\mathbb{R},L^r)}\,.
\end{align*}
Since by definition $D_t(e^{itH}f)=H(e^{itH}f)$ we deduce by the usual functional calculus that $\langle D_t\rangle ^{\frac{\varepsilon}{2}}(e^{itH}f)=\langle H\rangle ^{\frac{\varepsilon}{2}}(e^{itH}f)$, which ends the proof. 
\end{proof}

\section{The lens transform}\label{3appendixLens}

The lens transform $\mathcal{L} : \mathcal{S}'(\mathbb{R} \times \mathbb{R}^d) \to \mathcal{S}'(\left(-\frac{\pi}{4},\frac{\pi}{4}\right) \times \mathbb{R}^d)$ and its inverse $\mathcal{L}^{-1}$ are defined by the following formulae
\[v(t,x):=\mathcal{L}u(t,x)=\frac{1}{\cos (2t)^{\frac{d}{2}}}u\left(\frac{\tan (2t)}{2},\frac{x}{\cos (2t)}\right) \exp \left(-\frac{i|x|^2\tan (2t)}{2}\right)\,,\]
\[u(s,y)=\mathcal{L}^{-1}v(s,y)=(1-4s^2)^{-\frac{d}{4}}v\left(\frac{\operatorname{arctan}(2s)}{2},\frac{y}{\sqrt{1-4s^2}}\right)\exp\left(\frac{i|y|^2s}{\sqrt{1-4s^2}}\right)\,.\] 
Formal computations show that $u$ solves \[i\partial _su+\Delta_yu=0 \text{ on } \mathbb{R}_s\times \mathbb{R}^d_y \text{ and } u(0)=u_0\] if and only if $v=\mathcal{L}u$ solves 
\[i\partial _tv-Hv=0 \text{ on } \left(-\frac{\pi}{4},\frac{\pi}{4}\right) \times \mathbb{R}^d_x \text{ and } v(0)=u_0\,.\]

With the variables $s=\frac{\tan (2t)}{2}$, or equivalently $t=t(s)=\frac{\operatorname{arctan}(2s)}{2}$, and $y=\frac{x}{\cos (2t)}$ an elementary computation shows that $\mathcal{L}$ maps solution of \eqref{3NLS} to solution of \eqref{3NLSH} with the same initial data. In particular \[\mathcal{L}(e^{is\Delta _y}u_0)=e^{-it(s)H}u_0\,.\] 

In the proof of Theorem~\ref{3main} it is needed to compare the $H^{\sigma}$ norms of $u$ and the $\mathcal{H}^{\sigma}$ norms of~$v$. This will be possible thanks to the following lemma. 

\begin{lemma}[Lens Transform]\label{3lens} If $u$ and $U$ are related by 
\[
    u(x)= \frac{1}{\cos^{\frac{d}{2}} (2t)} U\left( \frac{x}{\cos (2t)}\right)\exp \left(-i \frac{x^2 \tan (2t)}{2} \right)
\] 
then for any $\sigma \in [0,1]$ and $0 \leqslant |t| < \frac{\pi}{4}$, 
\[
    \|U\|_{H^{\sigma}} \lesssim \|u\|_{\mathcal{H}^{\sigma}}\,,\]\[\|\langle \cdot \rangle ^{\sigma}U\|_{L^2(\mathbb{R}^d)} \lesssim \left( \frac{\pi}{4} - t \right)^{-\sigma} \|u\|_{\mathcal{H}^{\sigma}(\mathbb{R}^d)}\,,
\]
\[
    \|u\|_{\mathcal{H}^{\sigma}} \lesssim \left(\frac{\pi}{4}-t\right)^{-\sigma}\|U\|_{H^{\sigma}}\,,
\]
and 
\[
    \|U\|_{L^q} \leqslant \cos (2t)^{d\left(\frac{1}{2}-\frac{1}{q}\right)}\|u\|_{L^q}\,.
\]
\end{lemma}

\begin{proof} See \cite{BT19}, Lemma~A.1 with only minor modification to dimension $d$. 
\end{proof}

\begin{remark} This result shows that estimates at regularity $\mathcal{H}^{\sigma}$ for $u$ transfers into estimates in $\mathcal{H}^{\sigma}$ for $U$ with a loss $\left(\frac{\pi}{4}-t\right)^{-2d\sigma}$. This explains that in the proof of Theorem~\ref{3main} we needed explicit decay rates on the scattering estimates for~\eqref{3NLSH} in order to be transferred into a scattering result for~\eqref{3NLS} . 
\end{remark}

\begin{remark}\label{3remSide} The last inequality of Lemma~\ref{3lens} applied to solutions $u$ of~\eqref{3NLS} and the corresponding solution $v$ to~\eqref{3NLSH} implies $\|u(s)\|_{L^q}\lesssim \|v(t(s))\|_{L^q}$ as soon as $q>2$.  
\end{remark}

\begin{remark} The ``lens transform'' may look surprising and somehow unexpected. We briefly explain how one can come up with such a transformation, only using basic insight regarding the symmetries of~\eqref{3NLS}. In fact this heuristic is useful to derive other symmetries or pseudo-symmetries to the Schrödinger equations, see~\cite{polyDanchin} for other instances of pseudo-symmetries and~\cite{merleRaphael} for an application.

To simplify, take $p=1+\frac{4}{d}$. Our starting point is the scaling symmetry 
\[
    u(s,y)\mapsto u_{\lambda}(s,y):=\frac{1}{\lambda^{\frac{d}{2}}}u\left(\frac{s}{\lambda^{2}},\frac{y}{\lambda}\right)\,.
\]
In order to derive more general symmetries we seek a scaling, depending on time. We change variables $(t,x) \leftrightarrow (s,y)$ imposing a local scaling symmetry $\mathrm{d}s = \frac{\mathrm{d}t}{\lambda(t)^2}$ and $\mathrm{d}y=\frac{\mathrm{d}x}{\lambda (t)}$ which can be integrated as \[y= \frac{x}{\lambda (s)} \text{ and } s= \int_0^{t} \frac{\mathrm{d}t'}{\lambda (t')^2}\,\cdotp\] 
Thus one may seek for a change of variable of the form: 
\[
    v(t,x)=\lambda (t)^{-d/2} u\left(\int_0^t\frac{\mathrm{d}t'}{\lambda (t')^2} , \frac{x}{\lambda (t)} \right)\,.
\] 
This change of variable is not sufficient. Indeed, writing the equation satisfied by $v$ in variables $t,x$ writes at the first derivatives order:
\[
    i\partial _t v + \Delta_xv - i \frac{\lambda '(t)}{\lambda (t)} x\cdot \nabla_xv + \cdots
\]
which is not the linear Schrödinger equation. However it is possible to multiply by the correct exponential to eliminate the term $x \cdot \nabla _x v$ just as in the method of variation of constants. This leads to the choice
\[
    v(t,x)=\lambda(t)^{-d/2} u\left(\int_0^t\frac{\mathrm{d}t'}{\lambda (t')^2} , \frac{x}{\lambda (t)} \right) \exp \left(\frac{i|x|^2 \lambda '(t)}{4\lambda (t)}\right)\,.
\]
Now if one wants to compactify time, a usual way to do so is to chose $t=\frac{1}{2}\operatorname{arctan}(2s)$ which maps $\mathbb{R}$ to $[-\frac{\pi}{4},\frac{\pi}{4}]$. Hence we arrive at the given form. 
\end{remark}

\bibliographystyle{alpha}
\bibliography{biblio}

\begin{thebibliography}{Dod16b}

\bibitem[AC09]{alazardCarles}
Thomas Alazard and R{\'e}mi Carles.
\newblock {Loss of regularity for supercritical nonlinear Schr{\"o}dinger
  equations}.
\newblock {\em Mathematische Annalen}, 343(2):397--420, 2009.

\bibitem[AT08]{ayacheTzvetkov}
A~Ayache and Nikolay Tzvetkov.
\newblock {{$L^p$} properties of Gaussian random series}.
\newblock {\em Trans. Amer. Math. Soc.}, 360:4425--4439, 2008.

\bibitem[Bar84]{barab}
J.E. Barab.
\newblock {Nonexistence of asymptotically free solutions for nonlinear
  Schrödinger equation}.
\newblock {\em J. Math. Phys.}, 25:3270--3273, 1984.

\bibitem[BOP15]{benyiOhPocovnicu}
\'Arp\'ad Bényi, Tadahiro Oh, and Oana Pocovnicu.
\newblock {Wiener randomization on unbounded domains and an application to
  almost sure local well-posedness of NLS}.
\newblock {\em Excursions in Harmonic Analysis}, 4:265--282, 2015.

\bibitem[Bou94]{bourgain}
Jean Bourgain.
\newblock {Periodic nonlinear Schrödinger equation and invariant measures}.
\newblock {\em Comm. Math. Phys.}, 166:1--26, 1994.

\bibitem[Bou96]{bourgain2d}
Jean Bourgain.
\newblock {Invariant measures for the 2D-defocusing nonlinear Schrödinger
  equation}.
\newblock {\em Comm. Math. Phys.}, 176:421--445, 1996.

\bibitem[Bou00]{bourgain00}
Jean Bourgain.
\newblock {Invariant Measures for NLS in Infinite Volume}.
\newblock {\em Communications in Mathematical Physics}, 210:605–--620, 2000.

\bibitem[Bri18]{bringmann2}
Bjoern Bringmann.
\newblock {{A}lmost sure scattering for the energy critical nonlinear wave
  equation}.
\newblock {\em arXiv:1812.10187, to appear in Amer. J. Math.}, 2018.

\bibitem[Bri20]{bringmann1}
Bjoern Bringmann.
\newblock {{A}lmost sure scattering for the radial energy critical nonlinear
  wave equation in three dimensions}.
\newblock {\em Analysis \& PDE}, 13:1011–1050, 2020.

\bibitem[BT08]{burqTzvetkov2}
Nicolas Burq and Nikolay Tzvetkov.
\newblock {Random data Cauchy theory for supercritical wave equations {I}:
  Local theory}.
\newblock {\em Inventiones Mathematicae}, 173:449--475, 2008.

\bibitem[BT20]{BT19}
Nicolas Burq and Laurent Thomann.
\newblock {{A}lmost Sure Scattering For The One Dimensional Nonlinear
  Schrödinger Equation}.
\newblock {\em \href{https://arxiv.org/abs/2012.13571}{arXiv:2012.13571}},
  2020.

\bibitem[BTT13]{burqThomannTzvetkov}
Nicolas Burq, Laurent Thomann, and Nikolay Tzvetkov.
\newblock {Long Time Dynamics for the One Dimensional Non Linear Schrödinger
  Equation}.
\newblock {\em Annales de l'Institut Fourier}, 63(6):2137--2198, 2013.

\bibitem[BTT15]{burq}
Nicolas Burq, Laurent Thomann, and Nikolay Tzvetkov.
\newblock {Global infinite energy solutions for the cubic wave equation}.
\newblock {\em Bulletin de la Société Mathématique de France},
  143(2):301--313, 2015.

\bibitem[Caz03]{cazenave}
T.~Cazenave.
\newblock {\em {Semilinear Schrödinger Equations}}.
\newblock Number~10 in Courant Lecture Notes in Mathematics. American
  Mathematical Society, Providence, RI, 2003.

\bibitem[CCT03]{christCollianderTao}
Michael Christ, James Colliander, and Terence Tao.
\newblock {Ill-posedness for nonlinear Schrödinger and wave equations}.
\newblock {\em arXiv math.AP/0311048}, 2003.

\bibitem[Den12]{deng}
Yu~Deng.
\newblock {Two dimensional nonlinear Schrödinger equation with random radial
  data}.
\newblock {\em Anal. PDE}, 5(5):913--960, 2012.

\bibitem[DLM19]{dodsonLuhrmannMendelson}
Benjamin Dodson, Jonas Lührmann, and Dana Mendelson.
\newblock {Almost sure local well-posedness and scattering for the 4D cubic
  nonlinear Schrödinger equation}.
\newblock {\em Adv. Math.}, 347:619--676, 2019.

\bibitem[Dod16a]{dodsonD1}
Benjamin Dodson.
\newblock {{G}lobal well-posedness and scattering for the defocusing, $L^2$
  critical, nonlinear Schrödinger equation when $d=1$}.
\newblock {\em Amer. J. Math.}, 138(2):531--569, 2016.

\bibitem[Dod16b]{dodsonD2}
Benjamin Dodson.
\newblock {{G}lobal well-posedness and scattering for the defocusing, $L^2$
  critical, nonlinear Schrödinger equation when $d=2$}.
\newblock {\em Duke Math. J.}, 165(18):3435--3516, 2016.

\bibitem[DR15]{polyDanchin}
R.~Danchin and P.~Raphaël.
\newblock Analyse non linéaire, cours de l’école polytechnique.
\newblock {\em math.univ-mlv.fr/users/danchin.raphael/publications.html}, 2015.

\bibitem[Dud02]{dudley}
R.M. Dudley.
\newblock {\em Real Analysis and Probability}.
\newblock Cambridge Studies in Advanced Mathematics. Cambridge University
  Press, 2002.

\bibitem[Hö65]{hormander}
L.~Hörmander.
\newblock Pseudo-differential operators and hypoelliptic equations.
\newblock {\em Proc. Symp. Pure math}, X:138--183, 1965.

\bibitem[IRT16]{imekrazRobertThomann}
Rafik Imekraz, Didier Robert, and Laurent Thomann.
\newblock {On random Hermite series}.
\newblock {\em Trans. AMS}, 368:2763--2792, 2016.

\bibitem[MR19]{merleRaphael}
Frank Merle and Pierre Raphaël.
\newblock {Blow up dynamic and upper bound on the blow up rate for critical
  nonlinear Schr\"odinger equation}.
\newblock {\em Annals of Mathematics}, 161:157--222, 2019.

\bibitem[Nak99]{nakanishi}
K.~Nakanishi.
\newblock {Energy scattering for nonlinear Klein-Gordon and Schrödinger
  equations in spatial dimensions $1$ and $2$}.
\newblock {\em J. Funct. Anal.}, 169:201--225, 1999.

\bibitem[OST18]{ohSosoeTzvetkov}
T.~Oh, P.~Sosoe, and N.~Tzvetkov.
\newblock {An optimal regularity result on the quasi-invariant Gaussian
  measures for the cubic fourth order nonlinear Schrödinger equation}.
\newblock {\em Journal de l’École polytechnique -- Mathématiques}, Tome
  5:793--841, 2018.

\bibitem[OT17]{OhTadahiTzvetkov}
Tadahiro Oh and N.~Tevetkov.
\newblock Quasi-invariant gaussian measures for the cubic fourth order
  nonlinear schrödinger equation.
\newblock {\em Probab. Theory Related Fields}, 169:1221--1168, 2017.

\bibitem[OT20]{ohTzvetkov1}
Tadahiro Oh and Nikolay Tzvetkov.
\newblock {Quasi-invariant Gaussian measures for the two-dimensional defocusing
  cubic nonlinear wave equation}.
\newblock {\em J. Eur. Math. Soc.}, 22:1785--1826, 2020.

\bibitem[Poi12]{poiret}
Aurélien Poiret.
\newblock {Solutions globales pour l'équation de Schrödinger cubique en
  dimension 3}.
\newblock {\em Preprint}, 2012.

\bibitem[PRT14]{poiretRobertThomann}
Aurélien Poiret, D.~Robert, and Laurent Thomann.
\newblock {Probabilistic global well-posedness for the supercritical nonlinear
  harmonic oscillator}.
\newblock {\em Anal. \& PDE}, 7:997--1026, 2014.

\bibitem[PW18]{pocovnicuWang}
Oana Pocovnicu and Yuzhao Wang.
\newblock {An $L^p$-theory for almost sure local well-posedness of the
  nonlinear Schrödinger equations}.
\newblock {\em Comptes Rendus -- Mathématique de l'Académie des Sciences},
  356:637--643, 2018.

\bibitem[Sze39]{szego}
Gabor Szegö.
\newblock {\em Orthogonal Polynomials}.
\newblock American Mathematical Society, 1939.

\bibitem[Tao06]{taoDispersive}
Terence Tao.
\newblock {\em {Nonlinear Dispersive Equations: Local and Global Analysis}}.
\newblock Conference Board of the Mathematical Sciences. Regional conference
  series in mathematics. American Mathematical Society, 2006.

\bibitem[Tay07]{taylor}
M.E. Taylor.
\newblock {\em Tools for PDE: Pseudodifferential Operators, Paradifferential
  Operators, and Layer Potentials}.
\newblock Mathematical surveys and monographs. American Mathematical Society,
  2007.

\bibitem[Tho09]{thomann}
Laurent Thomann.
\newblock {Random data Cauchy problem for supercritical Schrödinger
  equations}.
\newblock {\em Ann. I. H. Poincaré}, 26:2385--2402, 2009.

\bibitem[TVZ07]{taoScat}
Terence Tao, Monica Visan, and X.~Zhang.
\newblock {Global well-posedness and scattering for the defocusing mass -
  critical nonlinear Schrödinger equation for radial data in high dimensions}.
\newblock {\em Duke Math. Journal}, 140(1):165--202, 2007.

\bibitem[TY84]{tsutsumiYajima}
Y.~Tsutsumi and K.~Yajima.
\newblock {The asymptotic behavior of nonlinear Schrödinger equations}.
\newblock {\em Bull. Amer. Math. Soc.}, 11:186--188, 1984.

\bibitem[Tzv15]{tzvetkov1}
Nikolay Tzvetkov.
\newblock {Quasi-invariant Gaussian measures for one-dimensional Hamiltonian
  partial differential equations}.
\newblock {\em Forum Math. Sigma}, 3:35pp, 2015.

\end{thebibliography}
\end{document}